\documentclass{article}

\usepackage{graphicx,epsfig}
\usepackage{times}
\usepackage{hyperref}
\usepackage{amssymb, amsthm}
\usepackage{amssymb,amsmath}
\usepackage{pstricks}
\usepackage{authblk}

\newtheorem{thm}{Theorem}
\newtheorem{clm}{Claim}
\newtheorem{lem}{Lemma}
\newtheorem{obs}{Observation}
\newtheorem{defi}{Definition}
\author{Manu Basavaraju}
\author{L. Sunil Chandran}
\author{Deepak Rajendraprasad}
\author{Arunselvan Ramaswamy}
\affil
{
	Department of Computer Science and Automation, \authorcr
	Indian Institute of Science, \authorcr
	Bangalore -560012, India. \authorcr
	\{manu, sunil, deepakr, arunselvan\}@csa.iisc.ernet.in
}

\begin{document}
\pagestyle{plain}
\title{Rainbow Connection Number of Graph Power and Graph Products}
\maketitle
\begin{abstract}
\textit{Rainbow connection number}, $rc(G)$, of a connected graph G is the minimum number of colors needed to color its edges so that every pair of vertices is connected by at least one path in which no two edges are colored the same (Note that the coloring need not be proper). In this paper we study the rainbow connection number with respect to three important \emph{graph product} operations (namely {\em cartesian product}, {\em lexicographic product} and {\em strong product}) and the operation of taking the \emph{power of a graph}. In this direction, we show that if $G$ is a graph obtained by applying any of the operations mentioned above on non-trivial graphs, then $rc(G) \le 2r(G)+c$, where $r(G)$ denotes the radius of $G$ and $c \in \{0,1,2\}$. In general the rainbow connection number of a bridgeless graph can be as high as the square of its radius \cite{basavaraju2010rainbow}. This is an attempt to identify some graph classes which have rainbow connection number very close to the obvious lower bound of $diameter$ (and thus the $radius$). The bounds reported are tight upto additive constants. The proofs are constructive and hence yield polynomial time $(2+\frac{2}{r(G)})$-factor approximation algorithms.
\\
\textbf{Keywords:} Graph Products, Graph Power, Rainbow Coloring.
\end{abstract}
\section{Introduction}

{\em Edge colouring} of a graph is a function from its edge set to the set of natural numbers. A path in an edge coloured graph with no two edges sharing the same colour is called a {\em rainbow path}. An edge coloured graph is said to be {\em rainbow connected} if every pair of vertices is connected by at least one rainbow path. Such a colouring is called a {\em rainbow colouring} of the graph. The minimum number of colours required to rainbow colour a connected graph is called its {\em rainbow connection number}, denoted by $rc(G)$. For example, the rainbow connection number of a complete graph is $1$, that of a path is its length, and that of a star is its number of leaves. For a basic introduction to the topic, see Chapter $11$ in \cite{chartrand2008chromatic}.

The concept of rainbow colouring was introduced in \cite{chartrand2008rainbow}. It was shown in \cite{chakraborty2009hardness} that computing the rainbow connection number of a graph is NP-Hard. To rainbow colour a graph, it is enough to ensure that every edge of some spanning tree in the graph gets a distinct colour. Hence order of the graph minus one is an upper bound for rainbow connection number. Many authors view rainbow connectivity as one `quantifiable' way of strengthening the connectivity property of a graph \cite{caro2008rainbow, chakraborty2009hardness, krivelevich2010rainbow}. Hence tighter upper bounds on rainbow connection number for graphs with higher connectivity have been a subject of investigation. The following are the results in this direction reported in literature: Let $G$ be a graph of order $n$. If $G$ is 2-edge-connected (bridgeless), then $rc(G) \leq 4n/5 -1$ and if $G$ is 2-vertex-connected, then $rc(G) \leq \min\{2n/3, n/2 + O(\sqrt{n})\}$ \cite{caro2008rainbow}. This was very recently improved in \cite{chandran2011raincon}, where it was shown that if $G$ is $2$-vertex-connected, then $rc(G) \leq \lceil n/2 \rceil$, which is the best possible upper bound for the case. It also improved the previous best known upper bound for $3$-vertex connected graphs of $3(n+1)/5$ \cite{li2010rain3con}. It was shown in \cite{krivelevich2010rainbow} that $rc(G) \leq 20n/\delta$ where $\delta$ is the minimum degree of $G$. The result was improved in \cite{chandran2010raindom} where it was shown that $rc(G) \leq 3n/(\delta + 1) + 3$. Hence it follows that $rc(G) \leq 3n/(\lambda + 1) + 3$ if $G$ is $\lambda$-edge-connected and $rc(G) \leq 3n/(\kappa + 1) + 3$ if $G$ is $\kappa$-vertex-connected. It was shown in \cite{chandran2011raincon} that the above bound in terms of edge connectivity is tight up to additive constants and that the bound in terms of vertex connectivity can be improved to $(2 + \epsilon)n/\kappa + 23/ \epsilon^2$, for any $\epsilon > 0$.

Many, but not all, of the above bounds are increasing functions of $n$. Since diameter, and hence radius, are lower bounds for rainbow connection number, any upper bound which is a function of one of the lower bounds alone is of great interest. Apart from the structural insights that it gives to the problem, it can also have applications in the design and analysis of approximation algorithms for rainbow colouring, which is known to be an NP-hard problem \cite{chakraborty2009hardness}. For a general graph, the rainbow connection number cannot be upper bounded by a function of radius or diameter alone. For instance, the star $K_{1,n}$ has a radius $1$ but rainbow connection number $n$. Still, the question of whether such an upper bound exists for special graph classes remain.

A very general result in this direction is the one by Basavaraju et al. \cite{basavaraju2010rainbow} which says that for every bridgeless graph of radius $r$, the rainbow connection number is upper bounded by $r(r+2)$. They also demonstrate that the above bound, which is quadratic in the radius, is tight not just for bridgeless graphs but also for graphs of any higher connectivity. This result was extended to graphs with bridges in \cite{dong2011rainbow}. This throws open a few interesting questions. Which classes of graphs have upper bounds on rainbow connection number which is (1) constant factor of radius, (2) additive factor above radius, etc. It is evident that answers to these questions will help in the design and analysis of constant factor and additive factor approximation algorithms for the problem. Moreover, they can give hints to answering the still open question of characterising graphs for which the rainbow connection number is equal to the diameter. Such additive factor upper bounds were demonstrated for unit interval, interval, AT-free, circular arc, threshold and chain graphs in \cite{chandran2010raindom}. Basavaraju et. al \cite{basavaraju2010rainbow} also showed that rainbow connection number will have a constant factor upper bound on bridgeless graphs in which the size  of a maximum induced cycle (chordality) is bounded independently of radius.

In this paper, we demonstrate a large class of graphs for which the rainbow connection number is upper bounded by a linear function of its radius. We study the rainbow connection number with respect to three important \emph{graph product} operations (namely {\em cartesian product}, {\em lexicographic product} and {\em strong product}) and the operation of taking the \emph{power of a graph}. Specifically, we show that if $G$ is a graph obtained by applying any of the operations mentioned above on non-trivial graphs, then $rc(G) \le 2r(G)+c$, where $r(G)$ denotes the radius of $G$ and $c \in \{0,1,2\}$. The bounds reported are either tight or tight upto additive constants. See Section \ref{Results} for the exact statements. The proofs are constructive and hence yield polynomial time $(2+\frac{2}{r(G)})$-factor approximation algorithms.

~~~~~~

The rainbow connection number of some graph products has got recent attention \cite{li2011characterize,gologranc2011rainbow,klavzar2011rainbow}. One way to bound the rainbow connection number of a graph product is in terms of the rainbow connection number of the operand graphs. Such an approach was adopted by Li et al. \cite{li2011characterize} to study rainbow connection number with respect to Cartesian product and the strong product. In particular, they show that the rainbow connection number of the Cartesian product and hence the strong product of two connected graphs are upper bounded by the sum of the rainbow connection numbers of the operand graphs. Later, it was shown in \cite{gologranc2011rainbow} that the rainbow connection number of the strong product of two connected graphs is upper bounded by the larger of the rainbow connection numbers of the operand graphs. Most of the bounds mentioned above can be far from being tight when the rainbow connection number of the operand graphs is much higher than their radii. The importance of the bounds reported here is that they are independent of the rainbow connection number of the operand graphs and depends only on the radius of the resultant graph.

~~~~~

\subsection{Preliminaries}

The graphs considered in this paper are finite, simple and undirected. Given a graph $G$, $|G|$ denotes the number of vertices in the graph, also called the \textit{order} of $G$.
A \textit{trivial} graph is a graph of order 0 or 1.

Given a graph $G$, a \textit{walk} in $G$, from vertex $u$ to vertex $v$ is defined as a sequence of vertices (not necessarily distinct), starting at $u$ and ending at $v$, say
$(u=u_{0}),\ u_{1},\ \dots,\ (u_{k}=v)$ such that $(u_{i}, u_{i+1}) \in E(G)$ for $0 \le i \le k-1$. A \textit{walk} in which all the vertices are distinct is called a \textit{path}. The length of a path is the number of edges in that path. A single vertex is considered to be a path of length $0$. The \textit{distance} between two vertices $u$ and $v$ in $G$ is the length of a shortest path between them and is denoted by $dist_{G}(u, v)$. Given two walks $W_{1} = u_{0}, u_{2}, \dots, u_{k}$ and $W_{2} = v_{0}, v_{1}, \dots, v_{l}$ such that $u_{k} = v_{0}$, we can \textit{concatenate} $W_{1}$ and $W_{2}$ to get a longer walk, $W = W_{1} \ldotp W_{2} = u_{0}, u_{1}, \dots, (u_{k}=v_{0}), v_{1}, v_{2}, \dots, v_{l}$.

Given a graph $G$, the \textit{eccentricity} of a vertex, $v \in \ V(G)$ is given by $ecc(v) = \max\{ dist_{G}(v, u) : u \in V(G)\}$. The \textit{radius} of $G$ is given by $r(G) = \min\{ecc(v): v \in V(G)\}$ and the \textit{diameter} of $G$ is defined as $diam(G) = \max\{ecc(v): v \in V(G)\}$. A \textit{central vertex} of $G$ is a vertex with eccentricity equal to the radius of $G$.

\begin{defi}[The Cartesian Product]
Given two graphs $G$ and $H$, the Cartesian product of $G$ and $H$, denoted by $G \Box H$, is defined as follows: $V(G \Box H) = V(G) \times V(H)$. Two distinct vertices $[g_{1}, h_{1}]$ and $[g_{2}, h_{2}]$ of $G \Box H$ are adjacent if and only if either $g_{1} = g_{2}$ and $(h_{1}, h_{2}) \in E(H)$ or $h_{1} = h_{2}$ and $(g_{1}, g_{2}) \in E(G)$.
\end{defi}

\begin{defi}[The Lexicographic Product]
Given two graphs $G$ and $H$, the lexicographic product of $G$ and $H$, denoted by
$G \circ H$, is defined as follows: $V(G \circ H) = V(G) \times V(H)$. Two distinct vertices $[g_{1}, h_{1}]$ and $[g_{2}, h_{2}]$ of $G \circ H$ are adjacent if and only if either $(g_{1}, g_{2}) \in E(G)$ or $g_{1} = g_{2}$ and $(h_{1}, h_{2}) \in E(H)$.

\end{defi}

\begin{defi}[The Strong Product]
Given two graphs $G$ and $H$, the strong product of $G$ and $H$, denoted by $G \boxtimes H$, is defined as follows: $V(G \boxtimes H) = V(G) \times V(H)$. Two distinct vertices
$[g_{1}, h_{1}]$ and $[g_{2}, h_{2}]$ of $G \boxtimes H$ are adjacent if and only if one of the three conditions hold:

\begin{enumerate}
\item $g_{1} = g_{2}$ and $(h_{1}, h_{2}) \in E(H)$ or
\item $h_{1} = h_{2}$ and $(g_{1}, g_{2}) \in E(G)$ or
\item $(g_{1}, g_{2}) \in E(G)$ and $(h_{1}, h_{2}) \in E(H)$.
\end{enumerate}
\end{defi}

It is easy to see from the definitions of the products above that if $G = K_{1}$ (respectively $H=K_1$) then the resultant graph is isomorphic to $H$ (respectively $G$). The above graph products are extensively studied in graph theory. See \cite{imrich2000product} for a comprehensive treatment of the topic.

\begin{defi}[Power of a graph]
The {\em k-th Power } of a graph, denoted by $G^{k}$ where $k \ge 1$, is defined as follows: $V(G^{k}) = V(G)$. Two vertices $u$ and $v$ are adjacent in $V(G^{k})$ if and only if the distance between vertices $u$ and $v$ in $G$, i.e., $dist_{G}(u,v) \le k$.
\end{defi}

Given a graph $G$, another graph $G'$ is called a \textit{spanning subgraph} of $G$ if $G'$ is a subgraph of $G$ and $V(G') = V(G)$. A vertex $v$ is called \textit{universal} if it is adjacent to all the other vertices in the graph.

Given a tree $T$, the unique path between any two vertices, $u$ and $v$ in $T$ is denoted by $P_{T}(u, v)$. It is sometimes convenient to consider some vertex from the tree as special; such a vertex is then called the \textit{root} of this tree. A tree with a fixed root is called a \textit{rooted tree}.

Let $T$ be a \textit{rooted tree} with root $root(T) = v_{0}$. The \textit{level number} of any vertex $v \in T$ is given by $\ell_{T}(v) = dist_{T}(v,v_{0})$. If the tree in context is clear then we simply use $\ell(v)$. The \textit{depth} of $T$ is given by $d(T) = \max\{\ell(v): v \in V(T)\}$. Given two vertices $u, v \in V(T)$, $u$ is called an \textit{ancestor} of $v$ if $u \in P_{T}(v, v_{0})$. It is easy to see that $\ell(v) \ge \ell(u)$. If $u$ is an ancestor of $v$ and $\ell(v) = \ell(u) + 1$ then $u$ is called the \textit{parent} of $v$ and is denoted by $parent(v)$.

\begin{defi}[Layer-wise Coloring of a Rooted Tree]
 \label{ColorFunc}

Given a rooted tree $T$ and an ordered multi-set of colors $C = \{ c_{i}: 1 \le i \le n\}$ where $n \ge d(T)$, we define the edge coloring, $f_{T,C}: E(T)\rightarrow C$ as
$f_{T,C}((u,v)) = c_{i}$ where $i =\max\{ \ell(u), \ell(v)\}$. We refer to $f_{T,C}$ as the \textit{Layer-wise Coloring} of $T$ that uses colors from the set $C$.
\end{defi}
Given an edge coloring $f$ of a graph $G$ using colors from the set $C$, let $C' \subseteq C$. Consider a path in $G$ that is rainbow colored with respect to $f$. We call this path a $C'$\textit{-Rainbow-Path} if every edge of the path is colored only from the set $C'$.
\begin{obs}
 \label{RootPath}
Let $T$ be a rooted tree and $C = \{c_{1}, c_{2}, \dots, c_{n}\}$ be an ordered set of colors such $c_{i} \ne c_{j}$ for $i \ne j$ and $n \ge d(T)$. Let $f_{T,C}$ be the Layer-wise Coloring of $T$ using colors from $C$. If $u, v \in V(T)$ such that $u$ is an ancestor of $v$ in $T$, then $P_{T}(v,u)$ is a \textit{C-Rainbow-Path} with respect to the coloring $f_{T,C}$. In particular $P_{T}(v,u)$ is a $\{c_{\ell(u)+1}, c_{\ell(u)+2}, \dots, c_{\ell(v)}\}$\textit{-Rainbow-Path} with respect to $f_{T,C}$.
\end{obs}
Recall the definition of the Cartesian Product of two graphs $G$ and $H$, denoted by $G \Box H$. We define a decomposition of $G \Box H$ into edge disjoint subgraphs as follows:

\begin{defi}[(G,H)-Decomposition of $G \Box H$]
 \label{GHDecomposition}

Given graphs $G$ and $H$ with vertex sets $V(G) = \{g_{i}: 0 \le i \le |G|-1\}$ and $V(H) = \{h_{i}: 0 \le i \le |H|-1\}$ respectively. We define a decomposition of $G \Box H$ as follows:
\\
For $0 \le j \le |H|-1$, define induced subgraphs, $G_{j}$, with vertex sets, $V(G_{j}) = \{[g_{i}, h_{j}]: 0 \le i \le |G|-1\}$. Similarly, for $0 \le i \le |G|-1$, define induced subgraphs, $H_{i}$, with vertex sets, $V(H_{i}) = \{[g_{i}, h_{j}]: 0 \le j \le |H|-1\}$. Then we have the following:
\begin{enumerate}
\item For $0 \le j \le |H|-1$, $G_{j}$ is isomorphic to $G$ and for $0 \le i \le |G|-1$, $H_{i}$ is isomorphic to $H$.
\item For $0 \le i < j \le |H|-1$, $V(G_{i}) \cap V(G_{j}) = \emptyset$ and hence $E(G_{i}) \cap E(G_{j}) = \emptyset$.
\item For $0 \le k < l \le |G|-1$, $V(H_{k}) \cap V(H_{l}) = \emptyset$ and hence $E(H_{k}) \cap E(H_{l}) = \emptyset$.
\item For $0 \le j \le |H|-1$ and $0 \le i \le |G|-1$, $V(G_{j}) \cap V(H_{i}) = [g_{i}, h_{j}]$ and $E(G_{j}) \cap E(H_{i}) = \emptyset$.
\end{enumerate}
\noindent
We call $G_{1}, G_{2}, \dots, G_{|H|-1}, H_{1}, H_{2}, \dots, H_{|G|-1}$ as the
\textit{\textbf{(G,H)-Decomposition}} of $G \Box H$.
\end{defi}

~~~~~~

\subsection{Our Results}
\label{Results}
\begin{enumerate}
\item If $G$ is a connected graph then $r(G^{k}) \le rc(G^{k}) \le 2r(G^{k}) + 1$ for all $k \ge 2$. 
The upper bound is tight up to an additive constant of $1$.
Note that $r(G^{k}) = \left\lceil\frac{r(G)}{k}\right\rceil$.
[See Theorem \ref{PowerThm}, Section \ref{SectionPower}]
\item If $G$ and $H$ are two connected, non-trivial graphs then $r(G \Box H) \le rc(G \Box H)$ $\le$ $2r(G \Box H)$.
The bounds are tight. Note that $r(G \Box H) = r(G) + r(H)$.
[See Theorem \ref{cart}, Section \ref{SectionCart}]
\item Given two non-trivial graphs $G$ and $H$ such that $G$ is connected we have the following:
\begin{enumerate}
 \item {
If $r(G \circ H) \ge 2$ then
$r(G \circ H) \le rc(G \circ H) \le 2r(G \circ H)$. This bound is tight.
}
\item {
If $r(G \circ H) = 1$ then
$1 \le rc(G \circ H) \le 3$. This bound is tight.
}
\end{enumerate}
[See Theorem \ref{LexicoTheorem}, Section \ref{SectionLex}]
\item If $G$ and $H$ are two connected, non-trivial graphs then
$r(G \boxtimes H) \le rc(G \boxtimes H) \le 2r(G \boxtimes H) + 2$.
The upper bound is tight up to an additive constant $2$. Note that $r(G \boxtimes H) = max\{r(G), r(H)\}$.
[See Theorem \ref{StrongTheorem}, Section \ref{SectionStrong}]
\end{enumerate}

Most of the bounds available in literature for graph products are in terms of raibow connection number of the operand graphs and hence can be far from being tight when the rainbow connection number of the operand graphs is much higher than their radii. It may happen that $rc(G)$ or $rc(H)$ are very large whereas $rc(G \Box H)$, $rc(G \boxtimes H)$, etc. are very small in comparison. For example let $G = K_{1,n}$ and $H = K_{2}$ then by the result in \cite{li2011characterize}, $rc(G \Box H) \le n + 1$ and by the result in \cite{gologranc2011rainbow}, $rc(G \boxtimes H) \le n$. But our results show that $rc(G \Box H) \le 4$ and $rc(G \boxtimes H) \le 4$. This suggests that the rainbow connection number of product of graphs may be related to the radii of the operand graphs (and hence on the radius of the resultant graph) rather than on their rainbow connection numbers. The results reported here confirm that it is indeed the case. It may be noted that a similar case is true even for graph powers. That is, $rc(G^{k})$ is independent of $rc(G)$ and is upper-bound by a linear function of $r(G^{k})= \lceil \frac{r(G)}{k}\rceil$.

~~~~~~

\section{Rainbow Connection Number of the k-th Power of a Graph H}
\label{SectionPower}
\paragraph{}
For $k \ge$ 1, recall that the \textit{k-th power of a graph H}, denoted by $H^{k}$,
as follows: $V(H^{k})$ = $V(H)$ and any two vertices $u$ and $v$ $\in$ $V(H^{k})$
are adjacent if and only if $dist_{H}(u,\ v)$ $\le$ $k$. It is easy to verify that $r(H^{k}) = \left\lceil \frac{r(H)}{k}\right\rceil$
and $diam(H^{k}) = \left\lceil \frac{diam(H)}{k}\right\rceil$.
\paragraph{}
Since $H^{1} = H$, for the remainder of this section we assume that $k \ge 2$.
Let $T$ be the \textit{BFS-Tree} rooted at some central vertex, say $h_{0}$, of $H$. 
Then clearly the depth of tree T, $d(T)$ $=$ $r(H)$.
Clearly $T^{k}$ is a spanning subgraph of $H^{k}$ and hence $rc(H^{k}) \le rc(T^{k})$.
So in order to derive an upper bound for $rc(H^{k})$ in terms of $r(H^{k})$ it is enough to derive an upper bound for $rc(T^{k})$
in terms of $\left\lceil \frac{d(T)}{k}\right\rceil$ ( $r(H^{k}) = \left\lceil \frac{d(T)}{k}\right\rceil$). 
\paragraph{}
Let $V(T)$ = \{ $h_{i}$: 0 $\le$ $i$ $\le$ $|H|-1$ \}. For $0 \le i \le k - 1$, let
$V_{i}$ = \{ $u \in V(T)$ : $\ell_{T}(u) > 0$ and $\ell_{T}(u) \equiv i\mod k$\}. It is easy to see that
$V$ = $\biguplus^{k-1}_{i=0}V_{i}$ 
$\uplus \{h_{0}\}$.
\\
\indent
For $0 \le i \le k-1$ and $0 \le j \le \left\lceil\frac{d(T)}{k}\right\rceil $ we define $V^{j}_{i} = \{u \in V_{i} \cup \{h_{0}\}
: \left\lceil\frac{\ell_{T}(u)}{k}\right\rceil = j\}$. Note that if $u \in V(T) \setminus \{h_{0}\}$ 
then $u$ belongs to exactly one $V^{j}_{i}$
where $0 \le i \le k-1$ and $1 \le j \le \left\lceil\frac{d(T)}{k}\right\rceil$. For all
$0 \le i \le k-1$, vertex $h_{0}$ is the only vertex in $V^{0}_{i}$.
\\
\indent
Now we define a function, $par$: $V(T) \setminus \{h_{0}\} \rightarrow V(T)$ as follows:  
$\forall u \in V(T) \setminus \{h_{0}\}$, $par(u) = v$
such that if $u \in V^{j}_{i}$ then $v \in V^{j-1}_{i}$ and $(u,v) \in E(T^{k})$. 
Such a vertex $v$ always exists because of the following reasons: 
If $1 \le \ell_{T}(u) \le k$
then $u \in V^{1}_{i}$ for some $0 \le i \le k-1$;  
we may choose $v$ to be $h_{0}$ since $h_{0} \in V^{0}_{i}$ and $(h_{0}, u) \in E(T^{k})$. 
If $\ell_{T}(u) > k$ then we may choose $v$ to be the ancestor of $u$ in $T$
such that $\ell_{T}(v) = \ell_{T}(u)-k$. Then clearly $v \in V^{j-1}_{i}$ and $(u, v) \in E(T^{k})$.
\\
\indent
For $0 \le i \le k-1$, define graph $G_{i}$ with vertex set, $V(G_{i}) = V_{i} \cup \{h_{0}\}$ and edge set,
$E(G_{i}) = \{(u, par(u)): u \in V_{i}\}$. Since every vertex in $G_{i}$ has a path to $h_{0}$, the only vertex in $V^{0}_{i}$,
$G_{i}$
is connected. Moreover using the definition of the function $par$, it is easy to verify that $G_{i}$ does not contain any cycle.
Hence $G_{i}$ is a tree. For $0 \le i \le k-1$ let $root(G_{i}) = h_{0}$. For $i \ne j$
we have $V(G_{i}) \cap V(G_{j}) = \{h_{0}\}$, a singleton set and hence $E(G_{i}) \cap E(G_{j}) = \emptyset$.
\paragraph{}
We define an edge coloring, $f$: $E(T^{k}) \rightarrow A \uplus B \uplus \{c\}$ where 
$A$ = \{ $a_{i}$: $1 \le i \le \lceil d(T)/k\rceil$ \},
$B$ = \{ $b_{i}$: $1 \le i \le \lceil d(T)/k\rceil$ \} and \{$c$\} are ordered sets of colors. 
Since $E(G_{i}) \cap E(G_{j}) =\emptyset$ for $i \ne j$, 
in order to define the edge coloring $f$
it is sufficient to define an edge coloring of $G_{i}$, for $0 \le i \le k-1$ and an edge coloring
 of all the remaining edges of $T^{k}$, separately.
For $0 \le i \le k-1$, if $i \equiv 0 \mod 2$ then
we choose the \textit{Layer-wise Coloring} $f_{G_{i}, A}$ to color the edges of
$G_{i}$ else we choose \textit{Layer-wise Coloring} $f_{G_{i}, B}$ to color the edges of $G_{i}$. All the remaining edges of $T^{k}$ are colored $c$.
\begin{clm}
 \label{PowerClaim}
The edge coloring $f$ is a \textit{rainbow coloring} of $T^{k}$
\end{clm}
\begin{proof}
Let $u$ and $v$ be two distinct vertices of $T^{k}$.
Without loss of generality let $u \ne h_{0}$. Then $u \in G_{i}$ where $0 \le i \le k-1$. By \textit{Observation }\ref{RootPath} there is an
\textit{A-Rainbow-Path} (\textit{B-Rainbow-Path}) from $u$ to $h_{0}$ if $i$ is \textit{even} (\textit{odd}). Now we can assume that
$u, v \ne h_{0}$.
Let $u \in V(G_{i})$ and $v \in V(G_{j})$.
To illustrate a rainbow path between $u$ and $v$ we consider the following two cases.
\\
\\
\textit{\textbf{Case 1:}} \textbf{[When $|i - j| \equiv 1 \mod 2$]}
\\
Without loss of generality let $i \equiv 0 \mod 2$ and $j \equiv 1 \mod 2$. 

Let $Q_{1}$ = $P_{G_{i}}(u,\ h_{0})$ and $Q_{2} = P_{G_{j}}(h_{0}, v)$ be the
\textit{A} and \textit{B-Rainbow-Paths} in $G_{i}$ and $G_{j}$ with respect to the \textit{Layer-wise Colorings} $f_{G_{i},A}$
and $f_{G_{j},B}$ respectively (See \textit{Observation }\ref{RootPath}). It follows that $Q_{1}$ and $Q_{2}$ are
\textit{A} and \textit{B-}\textit{Rainbow-Paths} in $T^{k}$ with respect to edge coloring $f$.
Clearly $Q$ = $Q_{1}$.$Q_{2}$ is
a \textit{(A $\cup$ B)-Rainbow-Path} from vertex $u$ to vertex $v$.
\\
\\
\textit{\textbf{Case 2:}} \textbf{[When $|i - j| \equiv 0 \mod 2$]}
\\
Without loss of generality we may assume that $\ell_{T}(v) \ge \ell_{T}(u)$.
\\
\indent
If $(u, v) \in E(T^{k})$ then there is a trivial rainbow path between them.
If $\ell_{T}(u_{1}) \le 1$ and $\ell_{T}(u_{2}) \le 1$ then ($u_{1}$, $u_{2}$) $\in$ $E(T^{k})$ (since $k \ge 2$).
We consider the case when $(u, v) \notin E(T^{k})$. This happens when
the level number of one of the vertices is $\ge 2$ i.e. 
$\ell_{T}(v) \ge 2$. Let $v_{1}$ $\in$ $V(T^{k})$ be the parent of  $v$ in $T$. Since $\ell_{T}(v) \ge 2$, $v_{1} \ne h_{0}$.
Let $v_{1} \in G_{l}$ where $\ell_{T}(v_{1}) = \ell_{T}(v)-1 \equiv l \mod k$.
From \textit{Case 1} we know that there is a $(A \cup B$)-\textit{Rainbow-Path}, say $P$,
between vertices $u$ and $v_{1}$ since $|i - l| \equiv 1 \mod 2$. 
Edge $(v, v_{1})$ is colored $c$ since $(v, v_{1}) \notin E(G_{i})$ for any $0 \le i \le k-1$.
Extending $P$ by edge $(v, v_{1})$ we get the required rainbow path between vertices $u$ and $v$.
\\
\noindent
We have thus proved that $f$ is a rainbow coloring of $T^{k}$.
\end{proof}
\begin{thm}
\label{PowerThm}
If $H$ is any connected, non-trivial graph then for all $k \ge 2$,  $r(H^{k})\le rc(H^{k}) \le$ \textit{2$r(H^{k})$ + 1}.
\end{thm}
\begin{proof}
The edge coloring $f$ uses $|A| + |B| + |\{c\}| = 2r(H^{k})+1$ colors.
The upper bound follows from \textit{Claim} \ref{PowerClaim}. The lower bound is trivial.
\end{proof}

\noindent \textbf{Tight Example:} \newline
Let $H$ be a path on $2kr+1$ vertices. It is easy to see that $rc(H^k) \ge diam(H^k)=2r(H^k)$.

{
\section{Rainbow Connection Number of the Cartesian Product of Two Non-trivial Graphs $G'$ and $H'$}
\label{SectionCart}
}
\paragraph{}
Recall that the Cartesian product, $G' \Box H'$, of two graphs $G'$ and $H'$ is defined as follows: $V(G' \Box H')$ =
$V(G')\ \times \ V(H')$. Two distinct vertices [$g_{1}$, $h_{1}$] and [$g_{2}$, $h_{2}$] of $G' \Box H'$ are adjacent if \textbf{either}
$g_{1}$ = $g_{2}$ and ($h_{1}$, $h_{2}$) $\in E(H')$ \textbf{or} ($g_{1}$, $g_{2}$) $\in E(G')$ and $h_{1} =\ h_{2}$. It is easy
to verify that $diam(G' \Box H') = diam(G') + diam(H')$ and that $r(G' \Box H') = r(G') + r(H')$. See \cite{imrich2000product} for proof.
\paragraph{}
Let \textit{G} be the \textit{Breadth-First-Search-Tree} (\textit{BFS-Tree}) rooted at some central vertex, say $g_{0}$, of 
\textit{$G'$}. Similarly let \textit{H} be the \textit{BFS-Tree} rooted at some central vertex, say $h_{0}$, of \textit{$H'$}. 
We have that \textit{d(G)} = \textit{r($G'$)} and \textit{d(H)} = \textit{r($H'$)}
where $d(G)$ and $d(H)$ are the depths of trees $G$ and $H$ respectively.
Clearly $G \Box H$ is 
a connected spanning subgraph of $G'\Box H'$ and therefore $rc(G' \Box H')$ $\le$ $rc(G \Box H)$. 
So in order to derive an upper bound for $rc(G' \Box H')$ in terms of $r(G' \Box H')$
it is sufficient to derive an upper bound for $rc(G \Box H)$
in terms of $r(G' \Box H')$.
\paragraph{}
Let $V(G)$ = 
\{ $g_{0}$, $g_{1}$, $\dots$, $g_{|G|-1}$\} and $V(H)$ = \{ $h_{0}$, $h_{1}$, $\dots$, $h_{|H|-1}$\}. 
Let $G_{1}, \dots G_{|H|-1}$, $H_{1}, \dots,$ $H_{|G|-1}$ be the \textit{(G,H)-Decomposition} of $G \Box H$ as defined in \textit{Definiton-}\ref{GHDecomposition}.
For $0 \le i \le |H|-1$ define $root(G_{i}) = [g_{0}, h_{i}]$ and for $0 \le j \le |G|-1$ define $root(H_{j}) = [g_{j}, h_{0}]$.
\\
Recall the following simple observations.
\begin{obs}
 \label{CartVert}
$V(G_{i}) \cap V(H_{j})$ = \{$[g_{j},\ h_{i}]$\},
$V(G_{i}) \cap V(G_{j})$ = $\emptyset$ and $V(H_{i}) \cap V(H_{j})$ = $\emptyset$, for all $i \neq j$.
\end{obs}
\begin{obs}
 \label{edgu}
 $E(G \Box H) = \biguplus^{|H|-1}_{i=0} E(G_{i}) \biguplus^{|G|-1}_{j=0} E(H_{j})$
\end{obs}
\paragraph{}
We now define an edge coloring, $f$: $E(G \Box H)$ $\rightarrow$ $A \uplus B \uplus C \uplus D$ where
$A$ = \{ $a_{i}$ :$\ $1 $\le$ $i$ $\le$ $d(G)$ \},
$B$ = \{ $b_{i}$ :$\ $1 $\le$ $i$ $\le$ $d(G)$ \},
$C$ = \{ $c_{i}$ :$\ $1 $\le$ $i$ $\le$ $d(H)$ \} and
$D$ = \{ $d_{i}$ :$\ $1 $\le$ $i$ $\le$ $d(H)$ \} are ordered sets of colors. 
In view of \textit{Observation-}\ref{edgu} it is clear that 
in order to define the coloring $f$, it is sufficient to 
describe separately, an edge coloring for each
$G_{i}$,
0 $\le$ $i$ $\le$ $|H|-1$ and
an edge coloring for each $H_{j}$, 0 $\le$ $j$ $\le$ $|G|-1$.
We choose Layer-wise Coloring $f_{G_{0},A}$ to be the edge coloring
of $G_{0}$ and $f_{G_{i},B}$ to be the edge coloring of $G_{i}$ for 1 $\le$ $i$ $\le$ $|H|\ -\ 1$. Similarly we choose Layer-wise Coloring
$f_{H_{0},C}$ to be the edge coloring
of $H_{0}$ and $f_{H_{i},D}$ to be the edge coloring of $H_{i}$ for 1 $\le$ $i$ $\le$ $|G|\ -\ 1$.
\begin{clm}
\label{CartProClaim}
The edge coloring, $f$, is a rainbow coloring of $G \Box H$.
\end{clm}
\begin{proof}
Let $u = [g_{i}, h_{j}]$ and $v = [g_{k}, h_{l}]$ be two distinct vertices of $G \Box H$. We demonstrate
a rainbow path between $u$ and $v$, by considering the following cases:
\\
\\
\textit{\textbf{Case 1:}} \textbf{[At least one of the vertices belong to $V(G \Box H)$ \textbackslash ($V(G_{0}) \cup$ $ V(H_{0})$)]}
\\
\\
Without loss of generality let $v \in$ $V(G \Box H)$ \textbackslash ($V(G_{0}) \cup V(H_{0})$)
i.e. $l \ne 0$ and $k \ne 0$. 
We now consider the following
two \textit{sub-cases}.
\\
\\
\textit{Case 1.a:} [Vertex $u \notin V(G_{0})$,  hence $j \ne 0$]
\\
Vertex $v = [g_{k}, h_{l}]$ $\in$ $V(H_{k})$ and
$root(H_{k})$ = $[g_{k},\ h_{0}]$. Let $Q_{1}$ = $P_{H_{k}}(v,\ [g_{k},\ h_{0}])$, is a \textit{D-Rainbow-Path} in $H_{k}$
with respect to the coloring $f_{H_{k},D}$, by \textit{observation} \ref{RootPath}.
Similarly let $Q_{2} = $ $P_{G_{0}}([g_{k},\ h_{0}],\ [g_{0},\ h_{0}])$, $Q_{3}$ = $P_{H_{0}}([g_{0},\ h_{0}],\ [g_{0},\ h_{j}])$
and $Q_{4}$ = $P_{G_{j}}([g_{0},\ h_{j}],\ [g_{i},\ h_{j}])$ be \textit{A-, C- and B-Rainbow-Paths} in $G_{0},\ H_{0}$ and $G_{j}$ ($j \ne 0$)
respectively.
It follows that $Q_{1}, Q_{2}, Q_{3}$ and $Q_{4}$ are $D$-,$A$-,$C$- and $B$\textit{-Rainbow-Paths} in $G \Box H$ 
with respect to the coloring $f$.
Clearly $Q$ = $Q_{1}.$ $Q_{2}.$ $Q_{3}.$ $Q_{4}$ is a \textit{rainbow walk} from $v$ to $u$ in $G \Box H$
that contains a \textit{rainbow path} between them.
\\
\\
\textit{Case 1.b:} [Vertex $u \in V(G_{0})$, hence $u = [g_{i}, h_{0}]$]
\\
Vertex $v \in V_{G_{l}}$, let $Q_{1}$ = $P_{G_{l}}(v, [g_{0},\ h_{l}])$, is a
\textit{B-Rainbow-Path} in $G_{l}$ with respect to edge coloring $f_{G_{l},B}$, by \textit{Observation-}\ref{RootPath}. 
Similarly let $Q_{2}$ = $P_{H_{0}}([g_{0},\ h_{l}],\ [g_{0},\ h_{0}])$ and
$Q_{3}$ = 
$P_{G_{0}}([g_{0},\ h_{0}],\ [g_{i},\ h_{0}])$ be $C$- and $A$\textit{-Rainbow-Paths} in $H_{0}$ and
$G_{0}$ respectively. It follows that $Q_{1},\ Q_{2}$ and $Q_{3}$ are \textit{B-, C-} and \textit{A-Rainbow-Paths} in $G \Box H$
with respect to the coloring $f$.
Clearly $Q$ = $Q_{1}.$ $Q_{2}.$ $Q_{3}.$ is a \textit{rainbow walk} from $v$ to $u$ 
 in $G \Box H$ that contains a rainbow path between them.
\\
\\
\textit{\textbf{Case 2:}} \textbf{[Both the vertices belong to $V(G_{0}) \cup V(H_{0})$]}
\\
Without loss of generality let $v \neq [g_{0}, h_{0}]$. We consider the following 3 sub-cases:
\\
\\
\textit{Case 2.a:} [Both the vertices belong to $V(H_{0})$, hence $u = [g_{0}, h_{j}]$ and $v = [g_{0}, h_{l}]$]
\\ 
Vertex $v = [g_{0}, h_{l}] \in V(G_{l})$.
Let $v'$ = [$g_{k'}$, $h_{l}$] be another vertex in $G_{l}$
such that ($v$, $v'$) $\in$ $E(G_{l})$. 
The existence of $v'$ is guaranteed since $G'$ $\ne$ $K_{1}$.
Let $Q_{1}$ = $P_{G_{l}}(v,\ v')$ i.e. the single edge $(v, v')$
is a \textit{B-Rainbow-Path} in $G_{l}$ with respect to the coloring $f_{G_{l},B}$, noting that $l \ne 0$ by the assumption
that $v \ne [g_{0}, h_{0}]$. 
Similarly let $Q_{2}$ = $P_{H_{k'}}(v',\ [g_{k'}, h_{0}])$, 
$Q_{3}$ = $P_{G_{0}}([g_{k'}, h_{0}],\ [g_{0}, h_{0}])$ and 
$Q_{4}$ = $P_{H_{0}}([g_{0},\ h_{0}],\ [g_{0},\ h_{j}])$ be \textit{D-, A-} and \textit{C-Rainbow-Paths} in $H_{k'},\ G_{0}$
and $H_{0}$ respectively. It follows that $Q_{1},\ Q_{2},\ Q_{3}$ and $Q_{4}$ are 
\textit{B-, D-, A-} and \textit{C-Rainbow-Paths} in $G \Box H$ with respect to coloring $f$.
Clearly $Q$ = $Q_{1}.$ $Q_{2}.$ $Q_{3}.$ $Q_{4}.$ is a \textit{rainbow walk} from \textit{v} to \textit{u} 
 in $G \Box H$ that contains a 
\textit{rainbow path} between them.
\\
\\
\textit{Case 2.b:} [Both the vertices belong to $V(G_{0})$, hence $u = [g_{i}, h_{0}]$ and $v = [g_{k}, h_{0}]$]
\\
Vertex $v \in V(H_{k})$.
Let $v'$ = [$g_{k}$, $h_{l'}$] be another vertex in $H_{k}$ such that ($v$, $v'$) $\in$ $E(H_{k})$.
The existence of $v'$ is guaranteed since $H' \ne K_{1}$.
Let $Q_{1}$ = $P_{H_{k}}(v,\ v')$ i.e. the single edge $(v, v')$
is a \textit{D-Rainbow-Path} in $H_{k}$ with respect to the coloring $f_{H_{k},D}$,
noting that $l \ne 0$ by the assumption that $v \ne [g_{0}, h_{0}]$. 
Similarly let $Q_{2}$ = 
$P_{G_{l'}}(v',\ [g_{0},\ h_{l'}])$, 
$Q_{3}$ = $P_{H_{0}}([g_{0},\ h_{l'}],\ [g_{0},\ h_{0}])$ and 
$Q_{4}$ = $P_{G_{0}}([g_{0},\ h_{0}], [g_{i},\ h_{0}])$ be \textit{B-, C-} and \textit{A-Rainbow-Paths} in $G_{l'}$,
$H_{0}$ and $G_{0}$ respectively. It follows that $Q_{1}, Q_{2}, Q_{3}$ and $Q_{4}$ are \textit{D-, B-, C-} and \textit{A-Rainbow-Paths}
in $G \Box H$ with respect to coloring $f$.
Clearly $Q$ = 
$Q_{1}.$ $Q_{2}.$ $Q_{3}.$ $Q_{4}.$ is a \textit{rainbow walk} from $v$ to $u$ 
 in $G \Box H$ that contains a \textit{rainbow path}
between them.
\\
\\
\textit{Case 2.c:} [One vertex belongs to $V(G_{0})$ and the other to $V(H_{0})$]
\\
Without loss of generality let $u \in V(G_{0})$, $v \in V(H_{0})$ then $j = 0$ and $l = 0$.
In view of \textit{Cases} $2.a$ and $2.b$ we can assume that $u, v \ne [g_{0}, h_{0}]$.

Let $Q_{1}$ = $P_{H_{0}}(v, [g_{0}, h_{0}])$ and $Q_{2}$ = $P_{G_{0}}([g_{0}, h_{0}], u)$
is a \textit{C-} and \textit{A-Rainbow-Paths} in $H_{0}$ and $G_{0}$ respectively.
It follows that $Q_{1}$ and $Q_{2}$ are $C$- and $A$\textit{-Rainbow-Paths} in $G \Box H$ with respect to the coloring $f$.
Clearly $Q$ = $Q_{1}. Q_{2}$ is a \textit{rainbow walk} from
vertex $v$ to vertex $u$ in $G \Box H$ that contains a \textit{rainbow path} between them.
\\
\\
It follows that $f$ is a rainbow coloring of $G \Box H$.
\end{proof}
\begin{thm}
\label{cart}
 If $G'$ and $H'$ are two non-trivial, connected graphs then $r(G' \Box H') \le rc(G' \Box H')$ $\le$ $2r(G' \Box H')$
\end{thm}
\begin{proof}
The edge coloring $f$ uses $|A| + |B| + |C| + |D|$ = $2(d(G) + d(H))$ = $2(r(G') + r(H'))$ = $2r(G' \Box H')$ number 
of colors. The upper bound follows from \textit{Claim-}\ref{CartProClaim} and the lower bound is obvious.
\end{proof}
\noindent
\textbf{Tight Example:}
\\
Consider two graphs $G_{1}$ and $G_{2}$ such that $diam(G_{1})$ = $2r(G_{1})$ and $diam(G_{2})$ = $2r(G_{2})$. For
example $G_{1}$ and $G_{2}$ may be taken as
paths with odd number of vertices. Then $diam(G_{1} \Box G_{2})$ = $diam(G_{1})$ + $diam(G_{2})$ = 
$2(rG_{1}) + r(H_{1}))$.
\section{Rainbow Connection Number of the Lexicographic Product of Two Non-trivial Graphs $G'$ and $H$}
\label{SectionLex}
Recall that the lexicographic product, $G' \circ H$, of two graphs $G'$ and $H$ is defined as follows: $V(G' \circ H)$ = $V(G') \times V(H)$.
Two distinct vertices [$g_{1}$, $h_{1}$] and [$g_{2}$, $h_{2}$] of $G' \circ H$  are adjacent if \textit{either} ($g_{1}$, $g_{2}$) $\in$
$E(G')$ \textbf{or} $g_{1}$ = $g_{2}$ and ($h_{1}$, $h_{2}$) $\in$ $E(H)$. 
Note that unlike the \textit{Cartesian Product} and the \textit{Strong Product}, the \textit{Lexicographic Product} is a
non-commutative product. Thus $G' \circ H$ need not be isomorphic to $H \circ G'$.
Also note that if $G'$ and $H$ are non-trivial graphs then $r(G' \circ H) = 1$ if and only if $r(G') = 1$ and $r(H) = 1$.
\begin{thm}
\label{LexicoTheorem}
Given two non-trivial graphs $G'$ and $H$ such that $G'$ is connected we have the following:
\begin{enumerate}
 \item {
\label{LexicoTheoremge2}
If $r(G' \circ H) \ge 2$ then
$r(G' \circ H) \le rc(G' \circ H) \le 2r(G' \circ H)$. This bound is tight.
}
\item {
\label{LexicoTheoremeq1}
If $r(G' \circ H) = 1$ then
$1 \le rc(G' \circ H) \le 3$. This bound is tight.
}
\end{enumerate}
\end{thm}
\subsection*{Part 1: $r(G' \circ H) \ge 2$}
Since $r(G' \circ H) \ge 2$, either $r(G') \ge 2$ or $r(H) \ge 2$. In either case it can be shown that $r(G' \circ H) \ge r(G')$.
Let $G$ be the \textit{BFS-Tree} rooted at some central vertex, say $g_{0}$, of graph $G'$. 
It is easy to see that the depth of $G$, \textit{d(G)} = $r(G')$.
Since $G \circ H$
is a connected spanning subgraph of $G' \circ H$, $rc(G' \circ H) \le rc(G \circ H)$. In order to derive an upper bound
for $rc(G' \circ H)$ in terms of $r(G' \circ H)$ 
it is sufficient to derive an upper bound for $rc(G \circ H)$ in terms of
$r(G' \circ H)$.
\paragraph{}
Let $V(G)$ = 
\{ $g_{i}$: 0 $\le$ $i$ $\le$ $|G| \ -\ 1$ \} and $V(H)$ = \{ $h_{i}$: 0 $\le$ $i$ $\le$ $|H| \ -\ 1$ \}.
Since $G$ is connected and non-trivial, vertex $g_{0}$ has at least one neighbor. We label this neighbor as $g_{1}$ 
i.e. $(g_{0}, g_{1}) \in E(G)$. 
Since $H$ is a non-trivial graph, there are at least two vertices in $H$ $-$ $h_{0}$ and $h_{1}$. 
Note that $(h_{0}, h_{1})$ need not be an edge in $H$.
It is easy to see that $G \Box H$ is a spanning subgraph of $G \circ H$. 
\paragraph
\indent
It is easy to see that $G \Box H$ is a spanning subgraph of $G \circ H$. 
Let $G_{0}, G_{1} \dots,$ $G_{|H|-1}$, $H_{0}, H_{1}, \dots, H_{|G|-1}$
be the \textit{(G,H)-Decomposition} of the subgraph of $G \circ H$ that is isomorphic to $G \Box H$ (See \textit{Definition }\ref{GHDecomposition}).
Recall that every $G_{i}$ is isomorphic to $G$ and every $H_{j}$ is isomorphic to $H$.
We define $root(G_{i}) =$ $[g_{0}, h_{i}]$ and $root(H_{j}) = [g_{j}, h_{0}]$. 
From \textit{Observation }\ref{CartVert} we know that any vertex 
$[g_{i}, h_{j}]$ belongs to both $G_{j}$ and $H_{i}$. 
\\
\\
\textit{\textbf{Special note on notation:}}
\\
In the rest of this section for any vertex
$v = [g_{i}, h_{j}] \in V(G_{j})$, we abuse the notation and simply use
$\ell(v)$ $/\ell([g_{i}, h_{j}])$ instead $\ell_{G_{j}}(v)$ $/\ell_{G_{j}}([g_{i}, h_{j}])$. Note that $\ell_{H}(v)$ need not make sense as 
$H$ need not be a tree.
\begin{defi}
\label{LexR2Part}
Let $E_{1} = \biguplus^{|H|-1}_{i=0}E(G_{i}) \ \biguplus^{|G|-1}_{j=0}E(H_{j})$
and $E_{2} = E(G \circ H) \setminus E_{1}$. 
\end{defi}
\noindent
We now define an edge coloring, $f: E(G \circ H) \rightarrow A \uplus B$ where $A = \{a_{i}: 1 \le i \le r(G' \circ H)\}$
and $B = \{b_{i}: 1 \le i \le r(G' \circ H)\}$ are ordered sets of colors. 
Since $r(G' \circ H) \ge 2$, both the sets $A$ and $B$ are of cardinality at least $2$.
Since $E(G \circ H) = E_{1} \uplus E_{2}$, it is enough to define separately
a coloring for $E_{1}$ and a coloring for $E_{2}$.
\\
\\
\textbf{\textit{Coloring the edges of $E_{1}$}}:
\\
\indent
To define a coloring of $E_{1}$ it is enough to define an edge colorings for each $G_{i}$, $0 \le i \le |H|-1$
and an edge coloring for each $H_{j}$, $0 \le j \le |G|-1$.
We choose the  
 \textit{Layer-wise Coloring}, $f_{G_{0},A}$ (as defined in \textit{Definition} \ref{ColorFunc})
to color the edges of $G_{0}$. 

\indent
We define a new ordered set, $B' = \{b'_{i}: 1 \le i \le r(G' \circ H)\}$ where $b'_{1} = a_{r(G' \circ H)} \in A$
and for $2 \le i \le r(G' \circ H)$, $b'_{i} = b_{i} \in B$.
For $1 \le i \le |H| - 1$, we choose the \textit{Layer-wise Coloring} $f_{G_{i}, B'}$ to be the edge coloring of $G_{i}$.
For $0 \le j \le |G| - 1$, we color all the edges of $H_{j}$ using the color $b_{1}$.
\\
\\
\textit{\textbf{Coloring the edges of $E_{2}$}}:
\\
For any vertex $v \in V(G \circ H)$ let $\mathcal{E}(v)$ be the set of edges from $E_{2}$ that are incident on $v$. We partition
$\mathcal{E}(v)$ into two sets $\mathcal{E}_{L}(v)$ and $\mathcal{E}_{U}(v)$. 
Consider some edge $(v, u) \in \mathcal{E}(v)$, then $(v, u) \in \mathcal{E}_{L}(v)$ if and only if $\ell(u) > \ell(v)$ and
$(v, u) \in \mathcal{E}_{U}(v)$ if and only if $\ell(u) < \ell(v)$.
For two vertices $v_{1}$ and $v_{2} \in V(G \circ H)$ we have that
$(v_{1}, v_{2}) \in \mathcal{E}_{L}(v_{1})$ \textit{if and only if} $(v_{1}, v_{2}) \in \mathcal{E}_{U}(v_{2})$.
\\
To color the edges of $E_{2}$ we have the following set of rules:
\begin{enumerate}
 \item[]$Rule\ \#1:$ All the edges of $\mathcal{E}_{L}([g_{0}, h_{0}])$ are colored $b_{1}$.
\item[]$Rule\ \#2:$ For all $v \in V(G_{0}) \setminus [g_{0}, h_{0}]$, all the edges of $\mathcal{E}_{L}(v)$ are colored $a_{\ell(v)+1}$.
\item[]$Rule\ \#3:$ All the edges of $\mathcal{E}_{U}([g_{i}, h_{0}])$, where $\ell([g_{i}, h_{0}]) = 1$, are colored $b_{r(G' \circ H)}$.
\item[]$Rule\ \#4:$ All the edges of $\mathcal{E}_{L}([g_{0}, h_{1}])$ $\setminus$
$\{([g_{0}, h_{1}], [g_{i}, h_{0}]): \ell([g_{i}, h_{0}]) = 1\}$ are colored $a_{r(G' \circ H)}$.
\item[]$Rule\ \#5:$ For all $v \in V(G_{1}) \setminus [g_{0}, h_{1}]$, all the edges from $\mathcal{E}_{L}(v)$ are colored $b_{\ell(v)+1}$.
\item[]$Rule\ \#6:$ All the edges of $\mathcal{E}_{U}([g_{i}, h_{1}])$
$\setminus$ $\{([g_{i}, h_{1}], [g_{0}, h_{0}])\}$, where $\ell_{G}(g_{i}) = 1$, are colored $a_{r(G' \circ H)}$.
\item[]$Rule\ \#7:$ All the remaining edges of $E_{2}$ are colored $b_{1}$.
\end{enumerate}

\begin{clm}
 \label{LexClaim}
The coloring $f$ is a rainbow coloring of the edges of $G \circ H$.
\end{clm}
\begin{proof}
Let $u = [g_{i}, h_{j}]$ and $v = [g_{k}, h_{l}]$ be two distinct vertices of $G \circ H$ such that
$\ell(v) \ge \ell(u)$. 
We demonstrate a rainbow path
between them by considering the following cases.
\\
\\
\textbf{\textit{Case 1:}} \textbf{[When $\ell(v) \ge 2$]}
\\
First we make the following $3$ observations.
\paragraph
\noindent
\textbf{(a):} There exists an \textit{A-Rainbow-Path} from $v$ to the vertex $[g_{0}, h_{0}]$:
\\
\noindent
If $v \in V(G_{0})$, then the path
$P_{G_{0}}(v, [g_{0}, h_{0}])$ is an $A$\textit{-Rainbow-Path} in $G_{0}$ with respect to the edge coloring
$f_{G_{0},A}$ (See \textit{Observation} \ref{RootPath}). 
If $v \notin V(G_{0})$, then $\exists\ v_{1} \in V(G_{0})$ such that $\ell(v_{1}) \ge 1$, $\ell(v_{1}) = \ell(v) - 1$
and $(v_{1}, v) \in \mathcal{E}_{L}(v_{1})$.
Such a vertex always exists since we have assumed that $\ell(v) \ge 2$; $G,\ H$ are non-trivial graphs
and $G$ is connected.
Since $v_{1} \in V(G_{0})$ there is an 
\textit{A-Rainbow-Path} from $v_{1}$ to $[g_{0}, h_{0}]$ as explained earlier, let this path be $P$.
Specifically $P$ is a $\{a_{1}, a_{2}, \dots, a_{\ell(v_{1})}\}$\textit{-Rainbow-Path}. Since edge $(v_{1}, v)$ is colored $a_{\ell(v_{1})+1}$
by $Rule\ \#2$,
we can extend path $P$ by $(v_{1}, v)$ to get the required \textit{A-Rainbow-Path} from $v$ to $[g_{0}, h_{0}]$.
\paragraph
\noindent
\textbf{(b):} There exists a \textit{B-Rainbow-Path} from $v$ to the vertex $[g_{0}, h_{0}]$: 
\\
\noindent
If $v \in V(G_{1})$ then there exists an
ancestor of $v$, say $v_{2}$, in $G_{1}$ such that $\ell(v_{2}) = 1$. The path $P_{1} = P_{G_{1}}(v, v_{2})$
is a $\{b_{\ell(v)}, b_{\ell(v)-1}, \dots, b_{2}\}$\textit{-Rainbow-Path} from $v$
to $v_{2}$ with respect to the edge coloring $f_{G_{1},B'}$. 
The edge $(v_{2}, [g_{0}, h_{0}])$ is colored $b_{1}$ by $Rule\ \#1$.
We can extend $P_{1}$ by edge $(v_{2}, [g_{0}, h_{0}])$ to get the required
\textit{B-Rainbow-Path} from vertex $v$ to $[g_{0}, h_{0}]$.
If $v \notin V(G_{1})$, then there exists $v_{3}$ = $[g_{i'}, h_{1}] \in V(G_{1})$ such that
$(v, v_{3}) \in \mathcal{E}_{L}(v_{3})$. Since $v_{3} \in V(G_{1})$
as explained earlier there is a $\{b_{\ell(v)}, b_{\ell(v)-1}, \dots, b_{2}, b_{1}\}$\textit{-Rainbow-Path}, say $P_{2}$, from $v_{3}$ to $[g_{0}, h_{0}]$.
Since the edge $(v_{3}, v)$ is colored $b_{\ell(v_{3})+1}$ by $Rule\ \#5$, we can extend $P_{2}$ by $(v_{3}, v)$ to get the required
\textit{B-Rainbow-Path} from $v$ to $[g_{0}, h_{0}]$.
\paragraph
\noindent
\textbf{(c):} There exist both $\{b_{\ell(v)}, b_{\ell(v)-1}, \dots, b_{2}, a_{r(G' \circ H)}\}$ and
\\
$\{a_{\ell(v)}, a_{\ell(v)-1}, \dots, a_{2}, b_{r(G' \circ H)}\}$\textit{-Rainbow-Paths} from $v$ to any vertex in $V(H_{0}) \setminus \{[g_{0}, h_{0}]\}$:
\\
\noindent
Recall that $\ell(v) \ge 2$.
From observation $(a)$ it can be inferred that there is a $\{b_{\ell(v)}, b_{\ell(v)-1}, \dots, b_{2}\}$\textit{-Rainbow-Path} from $v$ to some vertex $v_{4} \in V(G_{1})$
such that $\ell(v_{4}) = 1$. For any $v_{5}$ $\in V(H_{0}) \setminus [g_{0}, h_{0}]$, the edge $(v_{4}, v_{5})$ is colored $a_{r(G' \circ H)}$ by $Rule\ \#6$
or by the \textit{Layer-wise Coloring} $f_{G_{1},B'}$ (whatever is applicable). This
implies that there is a $\{b_{\ell(v)}, b_{\ell(v)-1}, \dots, b_{2}, a_{r(G' \circ H)}\}$\textit{-Rainbow-Path} from vertex $v$ to any vertex in $V(H_{0})$
$\setminus \{[g_{0}, h_{0}]\}$.
\\
Similarly from observation $(b)$ it can be inferred that there is a $\{a_{\ell(v)}, a_{\ell(v)-1}, \dots, a_{2}\}$\textit{-Rainbow-Path} from vertex $v$ to some vertex 
$v_{6} \in V(G_{0})$ such that $\ell(v_{6}) = 1$. By $Rule\ \#3$ any vertex in $V(H_{0}) \setminus \{[g_{0}, h_{0}]\}$ is adjacent to $v_{6}$ and is colored $b_{r(G' \circ H)}$.
\paragraph
\noindent
Now consider the different cases involving vertex $u$. If $\ell(u) \ge 2$ then from observations $(a)$ and $(b)$ it follows that $u$ and $v$ are rainbow connected. If $\ell(u) = 0$
then from observation $(c)$ it follows that $u$ and $v$ are rainbow connected. Finally if $\ell(u) = 1$ then we know that $(u, [g_{0}, h_{0}])$ $\in E(G \circ H)$
and is colored either $a_{1}$ or $b_{1}$. Since $v$ has both an $A$ and a $B$\textit{-Rainbow-Path} to $[g_{0}, h_{0}]$. It follows that $u$
and $v$ are rainbow connected.
%
%
\\
\\
\noindent
\textbf{\textit{Case 2:}} \textbf{[When $\ell(v) \le 1$]}
\\
Without loss of generality we assume that vertex $u \ne [g_{0}, h_{0}]$.
\\
\\
\textit{Case 2.a:} [When $\ell(v) \ne \ell(u)$]
\\
Vertices $u$ and $v$ are connected by an edge which is a trivial \textit{rainbow path} between them.
\\
\\
\textit{Case 2.b:} [When $\ell(v) = \ell(u) = 0$, hence $u = [g_{0}, h_{j}]$ and $v = [g_{0}, h_{l}]$]
\\
\indent
If $v = [g_{0}, h_{0}]$ then we claim that the two length path, $P = \{v= [g_{0}, h_{0}]$\}, $[g_{1}, h_{0}], \{[g_{0}, h_{j}] = u\}$
is a rainbow path from $v$ to $u$. The edges of $P$ are colored $a_{1}, b_{r(G' \circ H)}$ in that order.
To see this: edge $(v, [g_{1}, h_{0}]) \in E(G_{0})$ and $G_{0}$ is edge colored using the \textit{Layer-wise Coloring}, $f_{G_{0},A}$. It follows that the edge
is colored $a_{1}$ (See \textit{Observation }\ref{RootPath}). The edge
$([g_{1}, h_{0}], u) \in \mathcal{E}_{U}([g_{1}, h_{0}])$ and is colored $b_{r(G' \circ H)}$ by \textit{Rule $\#3$}.
Note that edge $(v, [g_{1}, h_{0}]) \in E(G \circ H)$ since $G$ is non-trivial and it is assumed that edge $(g_{0}, g_{1}) \in E(G)$.
\\
\indent
If $v \in V(H_{0}) \setminus \{[g_{0}, h_{0}]\}$ then we claim that the four length path, $P = \{u = [g_{0}, h_{j}]\},$ $[g_{1}, h_{0}], [g_{0}, h_{0}],$
$[g_{1}, h_{1}], \{[g_{0}, h_{l}] = v\}$ is a rainbow path from $u$ to $v$. The edges of $P$ are colored
$b_{r(G' \circ H)}$, $a_{1}, b_{1}$, $a_{r(G' \circ H)}$ in that order. To see this: edge 
$(u, [g_{1}, h_{0}]) \in \mathcal{E}_{U}([g_{1}, h_{0}])$ and is colored $b_{r(G' \circ H)}$ by \textit{Rule $\#3$};
edge $([g_{1}, h_{0}], [g_{0}, h_{0}]) \in E(G_{0})$ and is colored $a_{1}$ by the \textit{Layer-wise Coloring }$f_{G_{0},A}$;
edge
$([g_{0}, h_{0}], [g_{1}, h_{1}])\in \mathcal{E}_{L}([g_{0}, h_{0}])$ and is colored $b_{1}$ by $Rule\ \#1$; finally
edge $([g_{1}, h_{1}], v)$ is colored $a_{r(G' \circ H)}$ 
by one of the two applicable rules: \textit{(a):} Edge
$([g_{1}, h_{1}], v) \in E(G_{1})$ and $G_{1}$ is edge colored using the \textit{Layer-wise Coloring} $G_{G_{1},B'}$ or
\textit{(b):} Edge $([g_{1}, h_{1}], v) \in \mathcal{E}_{U}([g_{1}, h_{1}])$
$\setminus \{([g_{0}, h_{0}], [g_{1}, h_{1}])\}$ and is colored $a_{r(G' \circ H)}$ by $Rule\ \#4$.
\\
\\
\textit{Case 2.c:} [When $\ell(v) = \ell(u) = 1$]
\\
\indent
If exactly one of the vertices is in $G_{0}$.
Without loss of generality let $u \in V(G_{0})$ and $v \notin V(G_{0})$ then $u = [g_{i}, h_{0}]$ and $v = [g_{k}, h_{l\ne0}]$. 
We claim that the two length path $P = $
$\{u = [g_{i}, h_{0}]\}, [g_{0}, h_{0}], \{[g_{k}, h_{l}] = v\}$ is a rainbow path from vertex $u$ to vertex $v$. The edges of $P$ are 
colored $a_{1},\ b_{1}$ in that order.
\\
\indent
If $u, v \in V(G_{0})$ then $u = [g_{i}, h_{0}]$ and $v = [g_{k}, h_{0}]$. We claim that the four length path
$P =$ $\{u = [g_{i}, h_{0}]\}, [g_{0}, h_{0}],$ $[g_{1}, h_{1}], [g_{0}, h_{1}],$ $\{v=[g_{k}, h_{0}]\}$ is a rainbow path
from vertex $u$ to vertex $v$. The edges are colored $a_{1}, b_{1}, a_{r(G' \circ H)}, b_{r(G' \circ H)}$ in that order.
\\
\indent
If $u, v \notin V(G_{0})$ then $u = [g_{i}, h_{j\ne0}]$ and $v = [g_{k}, h_{l\ne0}]$.
We claim that the four length path
$P =$ $\{u = [g_{i}, h_{j}]\}, [g_{0}, h_{0}],$ $[g_{1}, h_{0}], [g_{0}, h_{1}],$ $\{v=[g_{k}, h_{l}]\}$ is a rainbow path
from $u$ to $v$. The edges of $P$ are colored $b_{1}, a_{1}, b_{r(G' \circ H)}, a_{r(G' \circ H)}$ in that order.
\paragraph
\noindent
We have thus proved that $f$ is a rainbow coloring of $G \circ H$. Since $f$ uses $2r(G \circ H)$ colors, we have $rc(G \circ H) \le 2r(G \circ H)$. Since it is assumed that $r(G \circ H) \ge 2$ we have proved the upper-bound in \textit{Part }\ref{LexicoTheoremge2} of \textit{Theorem }\ref{LexicoTheorem}.
\end{proof}
\noindent
\textbf{Tight Example:} 
\\
Let $G$ be a connected graph such that $r(G) \ge 2$ and $diam(G) = 2 r(G)$;
let $H$ be any non-trivial graph.
It is easy to see that $diam(G \circ H) = 
diam(G)$ and $r(G \circ H) = r(G)$. Hence we can conclude that $diam(G \circ H) = 2 r(G \circ H)$. We know that $rc(G \circ H) \ge diam(G \circ H)$ 
and $rc(G \circ H) \le 2 r(G \circ H)$ (\textit{Part }\ref{LexicoTheoremge2} from \textit{Theorem }\ref{LexicoTheorem}).
It follows that $rc(G \circ H) = 2 r(G \circ H)$.
\subsection*{Part 2: $r(G' \circ H) = 1$}
We know that if $r(G' \circ H) = 1$ then $r(G') = r(G) = 1$ \textit{and} $r(H) = 1$.

\begin{clm}
 \label{Lexicor1}
If $G'$ and $H$ are two non-trivial graphs such that $r(G' \circ H) = 1$ then $rc(G' \circ H) \le 3$.
\end{clm}
\begin{proof}
Since $r(G' \circ H) = 1$ there exists an universal vertex, say $u \in V(G' \circ H)$. It is easy to verify that $G' \circ H$
is $2$ vertex connected. Now consider the following theorem:
\\
\textbf{\textit{Theorem Chandran et al.}\cite{chandran2010raindom}:}
\textit{If $D$ is a connected two-way dominating set in a graph G, then $rc(G) \le rc(G[D]) + 3$.}
\\
The proof and definitions involved are given in \textbf{\cite{chandran2010raindom}}.
\paragraph
\indent
The universal vertex, $u$, is a trivial dominating set. Moreover since $G' \circ H$ is two vertex connected and consequently two edge connected, it follows that
$\{u\}$ is a two-way dominating set in $G' \circ H$. As a result $rc(G' \circ H) \le rc(\{u\})$ $+ 3$. Since $rc(\{u\}) = 0$ we have $rc(G' \circ H) \le \ 3$. We have thus proved
the \textit{claim} and the upper-bound in \textit{Part }\ref{LexicoTheoremeq1} of  \textit{Theorem }\ref{LexicoTheorem}.
\end{proof}

\noindent
\textbf{Tight Example:} 
\\
Consider two non-trivial graphs $G$ and $H$ such that $G = K_{1,n}$ (a \textit{star graph}) where $n \ge 2^{m}+1$ 
and $H$ is a graph such that $r(H) = 1$ and $|H| = m$.
We claim that $rc(G \circ H) = 3$.
\begin{proof}
We prove the claim by contradiction.
\\
\indent
Let $f$ be a rainbow coloring of $G \circ H$ using at most $2$ colors, say $a_{1}$ and $a_{2}$.
Let $V(G) = \{g_{0}, g_{1}, \dots, g_{n}\}$ where $g_{0}$ is the central vertex of $G$. Similarly let 
$V(H) = \{h_{0}, h_{1}, \dots, h_{m-1}\}$. Let $H_{0}$ be the induced subgraph of $G \circ H$ with vertex set $V(H_{0}) = $ $\{[g_{0}, h_{i}]: 0 \le i \le m-1\}$.
Graph $H_{0}$ is isomorphic to $H$.
\\
\indent
For $1 \le i \le n$ 
define the function $f_{i}: \{[g_{i}, h_{0}]\} \times V(H_{0}) \rightarrow \{a_{1}, a_{2}\}$ 
as
$f_{i}(([g_{i}, h_{0}],\ [g_{0}, h_{j}])) = f(([g_{i}, h_{0}],\ [g_{0}, h_{j}]))$. Each of the functions, $f_{i}$, are one among
$2^{|H|}$ possible functions. Since $n > 2^{|H|}$, by \textit{pigeon hole principle}
there must exist some $f_{i}$ and $f_{k}$ such that $i \ne k$ and $f_{i} = f_{k}$. If so there is \textit{no rainbow path} 
between the vertices $[g_{i}, h_{0}]$ and $[g_{k}, h_{0}]$ with respect to
the edge coloring $f$. 
This is beacause any rainbow path with respect to $f$ between the two vertices is of length $2$. 
Now any two length path between the vertices is of the form $[g_{i}, h_{0}], v, [g_{k}, h_{0}]$ where $v$ is the intermediate vertex.
It is easy to see that $v \in V(H_{0})$.
We know that $f_{i}([g_{i}, h_{0}], v)$
= $f_{k}([g_{k}, h_{0}], v)$ = $f([g_{i}, h_{0}], v)$ = $f([g_{k}, h_{0}], v)$ for all $v \in V(H_{0})$.
This is a contradiction.
Hence $f$ is not a \textit{rainbow coloring}
of $G \circ H$.
\\
Therefore any rainbow coloring of $G \circ H$ uses at least 3 colors. It follows from \textit{Claim} \ref{Lexicor1} that $rc(G \circ H) = 3$.
\end{proof}
\noindent
\textit{\textbf{Proof of Theorem \ref{LexicoTheorem}}:} The upper bounds follow from \textit{Claim }\ref{LexClaim} and \textit{Claim }\ref{Lexicor1}.
The lower bounds are trivial.
%
%
%
\section{Rainbow Connection Number of the Strong Product of Two Non-Trivial, Connected Graphs $G'$ and $H'$}
\label{SectionStrong}

Recall that the \textit{strong product} of two graphs $G'$ and $H'$, denoted by $G' \boxtimes H'$, is defined as follows: $V(G' \boxtimes H')$ = 
$V(G') \times V(H')$. The edge set of $G' \boxtimes H'$ consists of two types of edges. 
An edge $([g_{1}, h_{1}], [g_{2}, h_{2}])$ is
\textit{Type-1}  if and only if
\textbf{either} $g_{1} = g_{2}$ and $(h_{1}, h_{2})$ $\in E(H')$ \textbf{or} $h_{1} = h_{2}$ and $(g_{1}, g_{2})$ $\in E(G')$. The
edge is of \textit{Type-$2$} if and only if
$(g_{1}, g_{2})$ $\in E(G')$ and $(h_{1}, h_{2})$ $\in E(H')$. 
Let $r_{max}$ = max\{$r(G'), r(H')$\}.
It is easy to see that $r(G' \boxtimes H') = r_{max}$
and $diam(G' \boxtimes H') = max\{diam(G'), diam(H')\}$. See \textbf{\cite{imrich2000product}} for proof.
\paragraph{}
We assume without loss of generality that $r(G') \ge r(H')$ as
$G' \boxtimes H'$ is isomorphic to $H' \boxtimes G'$. 
Let $G$ and $H$ be \textit{BFS-Trees} rooted at some central vertices, $g_{0}$ and $h_{0}$ respectively of $G'$ and $H'$. 
It is easy to see that the depths of $G$ and $H$ are
$d(G) = r(G')$ and $d(H) = r(H')$ respectively. Let $V(G) = \{g_{i}: 0 \le i \le |G|-1\}$ and $V(H) = \{h_{i}: 
0 \le i \le |H|-1\}$. Since $G$ and $H$ are non-trivial connected trees there is atleast one neighbor for $g_{0}$
and $h_{0}$ in $G$ and $H$ respectively. 
In the remainder of the section we always let these vertices be $g_{1}$ and $h_{1}$ respectively.
Therefore $(g_{0}, g_{1}) \in E(G)$ and $(h_{0}, h_{1}) \in E(H)$.
\\
\indent
Let $L_{w}(G)$ = \{$g_{i} \in V(G)$: \textit{$\ell_{G}$($g_{i}$)} = $w$ \} for $0 \le w \le d(G)$ and
$L_{x}(H)$ = \{$h_{i} \in V(H)$: \textit{$\ell_{H}$($h_{i}$)} = $x$ \} for $0 \le x \le d(H)$.
We define $V_{w,x}$ = $L_{w}(G) \times L_{x}(H)$ for $0 \le w \le d(G)$ and $0 \le x \le d(H)$.
\paragraph{}
Since $G \boxtimes H$ is a spanning subgraph of $G' \boxtimes H'$, $rc(G' \boxtimes H') \le rc(G \boxtimes H)$.
So in order to derive an upper bound for $rc(G' \boxtimes H')$ in terms of $r(G' \boxtimes H')$ it is enough to derive
an upper bound for $rc(G \boxtimes H)$ in terms of $d(G) = r_{max} = r(G')$. Recall that we have assumed that $r(G') \ge r(H')$
and therefore $r(G' \boxtimes H') = r(G')$.
\\
\indent
We define an edge coloring, $f: E(G \boxtimes H) \rightarrow A \uplus B \uplus \{c, d\}$ where 
$A = \{a_{i}: 1 \le i \le d(G)\}$ and $B = \{b_{i}: 1 \le i \le d(G)\}$ are ordered sets of colors; and $c$ and $d$
are colors that are not in $A \uplus B$. Since $E(G \boxtimes H)$ is the disjoint union of \textit{Type-$1$} and \textit{Type-$2$}
edges, we can define the coloring for \textit{Type-$1$} and \textit{Type-$2$} edges separately.
\subsection*{Coloring the Type-1 edges}
Note that if we restrict the edge set of $G \boxtimes H$ to \textit{Type-1} edges alone then the subgraph thus obtained
is isomorphic to $G \Box H$, the Cartesian Product of $G$ and $H$.  
Let $G_{1}, G_{2}, \dots, G_{|H|-1}$, $H_{1}, H_{2}, \dots, H_{|G|-1}$ be the \textit{(G-H)-Decomposition} of $G \Box H$
(\textit{Type-$1$} edges) as defined in \textit{Definition }\ref{GHDecomposition}. For $0 \le j \le |H|-1$,
define $root(G_{j}) = [g_{0}, h_{j}]$ and for $0 \le i \le |G|-1$, define $root(H_{i}) = [g_{i}, h_{0}]$
\\
\\
Recall that $A = \{a_{i}: i \le i \le d(G)\}$ and $B = \{b_{i}: 1 \le i \le d(G)\}$ are ordered sets of colors. We define 
several new ordered (multi) sets of colors by slightly modifying the sets $A$ and $B$.
First we define the ordered set, $A_{0} = \{a^{0}_{i}: 1 \le i \le d(G)\}$ where $a^{0}_{1} = c$
and $a^{0}_{i} = a_{i} \in A$ for $2 \le i \le d(G)$.
Also for $1 \le w \le d(H)$, we define ordered multi-sets, 
$A_{w} = \{a^{w}_{i}: 1 \le i \le d(G)\}$ and $B_{w} = \{b^{w}_{i}: 1 \le i \le d(G)\}$
where $a^{w}_{i} = d$ and $b^{w}_{i} = d$ for $1 \le i \le w$ \textbf{and} $a^{w}_{i} = a_{i} \in A$
and $b^{w}_{i} = b_{i} \in B$
for $w+1 \le i \le d(G)$.
\\
\\
\noindent
\textbf{Rules to colors the \textit{Type-$1$} edges:}
\begin{itemize}
 \item[T1-R1:] We choose the \textit{Layer-wise Coloring} $f_{H_{0},A}$ to color the edges of $H_{0}$.
\item[T1-R2:] For each $H_{i}$ such that $\ell_{G}(g_{i}) = 1$, we choose the \textit{Layer-wise coloring} $f_{H_{i},B}$
to color the edges of $H_{i}$.
\item[T1-R3:] For each $H_{i}$ such that $\ell_{G}(g_{i}) \ge 2$, we color all the edges of $H_{i}$ using $d$.
\item[T1-R4:] For $0 \le w \le d(H)$ we choose $f_{G_{i},A_{w}}$ to color the edges of $G_{i}$ if $w$ is \textit{even}
and we choose $f_{G_{i}, B_{w}}$ to the color the edges of $G_{i}$ if $w$ is \textit{odd}.
\end{itemize}
\subsection*{Coloring the Type-$2$ edges}
\begin{obs}
 \label{CrossEdges}
If an edge $([g_{i}, h_{j}], [g_{k}, h_{l}]) \in E(G \boxtimes H)$ is of \textit{Type-$2$}
such that $[g_{i}, h_{j}] \in V_{w,x}$ and $[g_{k}, h_{l}] \in V_{y,z}$ then we have $|w - y| = 1$ and
$|x - z| = 1$.
\end{obs}
\begin{proof}
Since the edge $([g_{i}, h_{j}], [g_{k}, h_{l}])$ is of \textit{Type-$2$}, edges $(g_{i}, g_{k})$ and $(h_{j}, h_{l})$
are edges of trees $G$ and $H$ respectively. Therefore $|w - y| = |\ell_{G}(g_{i}) - \ell_{G}(g_{k})| = 1$
and $|x- z| = |\ell_{H}(h_{j}) - \ell_{H}(h_{l})| = 1$.
\end{proof}
\noindent
\textbf{Rules to colors the \textit{Type-$2$} edges:}
\begin{itemize}
 \item[T2-R1:]
Let $([g_{i}, h_{j}], [g_{k}, h_{l}]) \in E(G \boxtimes H)$ be an edge of \textit{Type-$2$} such that $[g_{i}, h_{j}] \in V_{y,z}$
and $[g_{k}, h_{l}] \in V_{y+1,z+1}$, then define
\begin{center}
$f(([g_{i}, h_{j}], [g_{k}, h_{l}]))$ =
$\begin{cases}
a_{z+1}\ if\ |z - y|\ is\ even \\
b_{z+1}\ if\ |z - y|\ is\ odd
\end{cases}$
\end{center}
Note that $z+1 = \ell_{H}(h_{l}) \le d(H) \le d(G)$ and therefore $a_{z+1}$ and $b_{z+1}$ exist.
\item[T2-R2:] Let $([g_{i}, h_{j}], [g_{k}, h_{l}]) \in E(G \boxtimes H)$ 
such that $[g_{i}, h_{j}] \in V_{1,1}$ and $[g_{k}, h_{l}] \in V_{2, 0}$ then we choose
$f(([g_{i}, h_{j}], [g_{k}, h_{l}])) = a_{2}$. 
\\
\\
Note that if $[g_{k}, h_{l}] \in V_{2,0}$ then $\ell_{G}(g_{k}) = 2$ and thus $d(G) \ge 2$ and $a_{2}$ exists.
\item[T2-R3:] All the remaining edges of \textit{Type-$2$} are colored $d$.
\end{itemize}

\noindent
\textbf{A-Reachable and B-Reachable Vertices:}
\\
We define the following $2$ concepts
with respect to the edge coloring $f$.
We define a vertex $[g_{i}, h_{j}] \in V(G \boxtimes H)$ to be \textit{A-Reachable} if there exists an \textit{A-Rainbow-Path} 
from $[g_{i}, h_{j}]$ to the
vertex $[g_{0}, h_{0}]$. We define $[g_{i}, h_{j}]$ to be \textit{B-Reachable} if there exists a \textit{B-Rainbow-Path}
from $[g_{i}, h_{j}]$ to some vertex in $V_{1,0}$.
\\
\\
We define two subsets, $R_{A}$ and $R_{B}$ of $V(G \boxtimes H)$:
\begin{eqnarray}
\label{StrongRA}
\nonumber
 R_{A} &=& \biguplus_{0 \le z \le d(H)}V_{0,z} \biguplus_{1 \le y \le z,\ |y-z|\ is\ even}V_{y,z} \biguplus_{2 \le y \le d(G)}V_{y,0} 
\biguplus_{2 \le z < y,\ z\ is\ even} V_{y,z} 
\\ \nonumber
R_{B} &=& \biguplus_{0 \le z \le d(H)}V_{1,z}\ \biguplus_{2 \le y \le z,\ |y-z|\ is\ odd}V_{y,z}\ 
\biguplus_{z < y,\ z\ is\ odd}V_{y,z}
\end{eqnarray}
\noindent
It is easy to verify that 
$R_{A} \cup R_{B} = V(G \boxtimes H)$, but $R_{A} \cap R_{B}$ is non-empty.
\begin{clm}
\label{rar}
 If $u \in R_{A}$, then $u$ is \textit{A-Reachable} with respect to the edge coloring $f$.
\end{clm}
\begin{proof}
Let $u = [g_{i}, h_{j}] \in V_{y,z}$. We consider the following $4$ cases.
\\
\\
\textit{\textbf{Case 1:}} \textbf{[When $u \in V_{0,z}$ where $0 \le z \le d(H)$]}
\\
From \textit{Rule T1-R1} we know that the edges of $H_{0}$ are colored using the \textit{Layer-wise Coloring}, $f_{H_{0}, A}$. Hence
by \textit{Observation }\ref{RootPath} there is an \textit{A-Rainbow-Path} from vertex $u$ to $root(H_{0}) = [g_{0}, h_{0}]$. It follows
that $u$ is \textit{A-Reachable}.
\\
\\
\textit{\textbf{Case 2:}} \textbf{[When $u \in V_{y,z}$ where $1 \le y \le z$ and $|y-z|$ is \textit{even}]}
\\
Since $\ell_{G}(g_{i}) = y$, the path from $g_{i}$ to $g_{0}$ in $G$ has $y+1$ vertices. 
Let this path be $g_{i}=g_{i_{0}}, g_{i_{1}}, \dots, g_{i_{y}}=g_{0}$. 
Let $h_{j'}$ be the ancestor of $h_{j}$ in $H$ such that $\ell_{H}(h_{j'}) = z-y$.
Let $h_{j}=h_{j_{0}}, h_{j_{1}}, \dots, h_{j'}=h_{j_{y}}$ be the path from $h_{j}$ to $h_{j'}$ in $H$. It has $y+1$ vertices. Clearly 
$P_{1} = \{[g_{i}, h_{j}] = [g_{i_{0}}, h_{j_{0}}]\}$, $[g_{i_{1}}, h_{j_{1}}], \dots, [g_{0}, h_{j'}]$ is a path in $G \boxtimes H$ whose edges are colored
$a_{z}, a_{z-1}, \dots, a_{z-y+1}$ in that order
(By \textit{Rule T$2$-R$1$}).
Note that if $y=z$ then $h_{j'} = h_{0}$ and $P_{1}$ is the required \textit{A-Rainbow-Path} from $u$ to $[g_{0}, h_{0}]$.
If $z<y$ then since $[g_{0}, h_{j'}] \in V(H_{0})$, by \textit{Case 1} there is a 
\textit{A-Rainbow-Path}, say $P_{2}$, from $[g_{0}, h_{j'}]$ to $[g_{0}, h_{0}]$. In particular
$P_{2}$ is a $\{a_{z-y}, a_{z-y-1}, \dots, a_{1}\}$ \textit{Rainbow Path}. Clearly $P = P_{1} \ldotp P_{2}$ is a
$\{a_{1}, a_{2}, \dots, a_{z}\}$\textit{-Rainbow-Path}
from vertex $u$ to $[g_{0}, h_{0}]$ with respect to coloring $f$. Hence $u$ is \textit{A-Reachable}.
\\
\\
\textit{\textbf{Case 3:}} \textbf{[When $u \in V_{y,0}$ where $2 \le y \le d(G)$, hence $u = [g_{i}, h_{0}] \in V(G_{0})$]}
\\
Let $u_{1} = [g_{i'}, h_{0}]$ be an ancestor of $u$ in $G_{0}$ such that $\ell_{G_{0}}(u_{1}) = 2$. By \textit{Rule T$1$-R$4$}
$G_{0}$ is edge colored using the \textit{Layer-wise Coloring} $f_{G_{0},A_{0}}$. The path
from vertex $u$ to $u_{1}$ in $G_{0}$, say $P_{1}$, 
is rainbow colored using colors from the set $\{a_{y}, a_{y-1}, \dots, a_{3}\}$.
Let $g_{i''}$ be the parent of $g_{i'}$ in $G$. Since $H$ is non-trivial $h_{1}$ exists and $(h_{0}, h_{1}) \in E(H)$. Therefore
$([g_{i'}, h_{0}], [g_{i''}, h_{1}]) \in E(G \boxtimes H)$ and is colored $a_{2}$ by \textit{Rule T$2$-R$2$}. 
Since $\ell_{G}(g_{i''}) = 1$, $(g_{i''}, g_{0}) \in E(G)$ and therefore $([g_{i''}, h_{1}], [g_{0}, h_{0}]) \in E(G \boxtimes H)$
and is colored $a_{1}$ by \textit{Rule T$2$-R$1$}. Hence the path
$P = P_{1}\ldotp ([g_{i'}, h_{0}],\ [g_{i''}, h_{1}],\ [g_{0}, h_{0}])$ is an \textit{A-Rainbow-Path} from vertex $u$ to $[g_{0}, h_{0}]$.
Hence $u$ is \textit{A-Reachable}.
\\
\\
\textit{\textbf{Case 4:}} \textbf{[When $u \in V_{y,z}$ where $y > z \ge 2$ and $z$ is \textit{even}]}
\\
Vertex $u = [g_{i}, h_{j}] \in V(G_{j})$. Let $u_{1} = [g_{i'}, h_{j}]$ be an ancestor of $u$ in $G_{j}$ such that
$\ell_{G_{j}}(u_{1}) = z$. 
Let $P_{1}$ be the path in $G_{j}$ from vertex $u$ to $u_{1}$.
Since $\ell_{H}(h_{j}) = z$ is \textit{even}, by \textit{Rule T$1$-R$4$},
$G_{j}$ is edge colored using the \textit{Layer-wise Coloring} $f_{G_{j},A_{z}}$.
The edges of $P_{1}$ are colored $a_{y}, a_{y-1}, \dots, a_{z+1}$ in that order. Since
$u_{1} = [g_{i'}, h_{j}] \in V_{z,z}$ and $z \ge 2$, by \textit{Case 2} we have a 
$\{a_{z}, a_{z-1}, \dots, a_{1}\}$\textit{-Rainbow-Path}, say $P_{2}$, from vertex $u_{1}$ to vertex $[g_{0}, h_{0}]$.
Clearly $P = P_{1}\ldotp P_{2}$ is an \textit{A-Rainbow-Path} from vertex $u$ to $[g_{0}, h_{0}]$. Hence vertex $u$ is
\textit{A-Reachable}.
\end{proof}
\begin{clm}
\label{rbr}
 If $u \in R_{B}$, then u is \textit{B-Reachable} with respect to the edge coloring $f$.
\end{clm}
\begin{proof}
 Let $u = [g_{i}, h_{j}] \in V_{y,z}$. We consider the following \textit{$3$} cases.
\\
\\
\textit{\textbf{Case 1:}} \textbf{[When $u \in V_{1,z}$ for $0 \le z \le d(G)$]}
\\
Vertex $u \in V(H_{i})$ with $root(H_{i}) = [g_{i}, h_{0}]$. Since $\ell_{G}(g_{i}) = 1$,
$H_{i}$ is edge colored using the \textit{Layer-wise Coloring} $f_{H_{i},B}$ by \textit{Rule T$1$-R$2$}. 
From \textit{Observation }\ref{RootPath}
we infer that there is a $\{b_{1}, b_{2}, \dots, b_{z}\}$\textit{-Rainbow-Path} from vertex $u$ to $[g_{i}, h_{0}] \in V_{1,0}$
in $H_{i}$. If follows that $u$ is \textit{B-Reachable} with respect to the edge coloring $f$.
\\
\\
\textit{\textbf{Case 2:}} \textbf{[When $u \in V_{y,z}$ where $2 \le y \le z$ and $|y-z|$ is \textit{odd}]}
\\
Let $u = [g_{i}, h_{j}] \in V_{y,z}$. In $G$ let $g_{i'}$ be the ancestor of $g_{i}$ with $\ell_{G}(g_{i'}) = 1$. Since 
$\ell_{G}(g_{i}) = y$, the path in $G$ from $g_{i}$ to $g_{i'}$ in $G$ has $y$ vertices. 
Let $g_{i}=g_{i_{0}}, g_{i_{1}}, \dots, g_{i_{y-1}}=g_{i'}$
be that path. Similarly in $H$ let $h_{j'}$ be the ancestor of $h_{j}$ with $\ell_{H}(h_{j'}) = z-y+1$. Then the path
in $H$ from $h_{j}$ to $h_{j'}$ has $y$ vertices. Let $h_{j}=h_{j_{0}}, h_{j_{1}}, \dots, h_{j_{y-1}}=h_{j'}$ be that path. Clearly
$P_{1} = [g_{i}, h_{j}], [g_{i_{1}, h_{j_{1}}}], \dots, [g_{i'}, h_{j'}]$ is a path in $G \boxtimes H$ and its edges are colored
$b_{z}, b_{z-1}, \dots, b_{z-y+2}$ in that order (By \textit{Rule T$2$-R$1$}). Now $[g_{i'}, h_{j'}] \in V_{1, z-y+1}$
and by \textit{Case 1} there is a $\{b_{1}, b_{2}, \dots, b_{z-y+1}\}$\textit{-Rainbow-Path}, say $P_{2}$, from $[g_{i'}, h_{j'}]$
to $[g_{i'}, h_{0}] \in V_{1,0}$. Clearly $P = P_{1}\ldotp P_{2}$ is a \textit{B-Rainbow-Path} from $u$ to 
$[g_{i'}, h_{0}] \in V_{1,0}$. It follows that $u$ is \textit{B-Reachable} with respect to the edge coloring $f$.
\\
\\
\textit{\textbf{Case 3:}} \textbf{[When $u \in V_{y,z}$ where $y > z$ and $z$ is odd]}
\\
Let $u = [g_{i}, h_{j}] \in V_{y,z}$. We consider the following two sub-cases.
\\
\\
\textit{Case 3.a:} [When \textit{y = z + 1}]
\\
Since $\ell_{H}(h_{j}) = z$, the path from $h_{j}$ to $h_{0}$ in $H$ has $z+1$ vertices. Let this path be 
$h_{j}=h_{j_{0}}, h_{j_{1}}, \dots, h_{j_{z}}=h_{0}$.
Similarly let $g_{i'}$ be the ancestor of $g_{i}$ in $G$ such that $\ell_{G}(g_{i'}) = 1$. 
Since $\ell_{G}(g_{i}) = z+1$ the path from $g_{i}$ to $g_{i'}$ in $G$ has $z+1$ vertices. Let this path be 
$g_{i}=g_{i_{0}}, g_{i_{1}}, \dots, g_{i_{z}}=g_{i'}$. Clearly $u=[g_{i}, h_{j}], [g_{i_{1}}, h_{j_{1}}], \dots, [g_{i'}, h_{0}]$ is a path
in $G \boxtimes H$ and is colored $b_{z}, b_{z-1}, \dots, b_{1}$ in that order (By \textit{Rule T$2$-R$1$}). 
Since $[g_{i'}, h_{0}] \in V_{1,0}$ vertex $u$ is \textit{B-Reachable}.
\\
\\
\textit{Case 3.b:} [When \textit{y $>$ z + 1}]
\\
Vertex $u \in G_{j}$. Let $u_{1} = [g_{i''}, h_{j}]$ be an ancestor of $u$ in $G_{j}$ such that $\ell_{G_{j}}(u_{1}) = z+1$.
Since $z$ is odd, by \textit{Rule T$1$-R$4$} we know that $G_{i}$ is edge colored usiong the 
\textit{Layer-wise Coloring} $f_{G_{j}, B_{z}}$.
The edges of path, $P_{1} = P_{G_{j}}(u, u_{1})$ are colored $b_{y}, b_{y-1}, \dots, b_{z+2}$
in that order and is a rainbow path.
Since $u_{1} \in V_{z+1,z}$ by \textit{Case 3.a} there is a $\{b_{z}, b_{z-1}, \dots, b_{1}\}$\textit{-Rainbow-Path}, say $P_{2}$,
from vertex $u_{1}$ to some vertex, say $u_{2}$ in $V_{1,0}$. Clearly $P = P_{1}\ldotp P_{2}$ is a \textit{B-Rainbow-Path} from vertex $u$
to $u_{2} \in V_{1,0}$. It follows that $u$ is \textit{B-Reachable} with respect to the coloring $f$.
\end{proof}
\noindent
\begin{clm}
 \label{DLink}
Let $u \in V(G \boxtimes H) \setminus \{[g_{0}, h_{0}]\}$ then we have the following:
\\
\textbf{(a)} If $u \in R_{A} \setminus R_{B}$ then there exists $u_{1} \in R_{B}$ such that 
$(u, u_{1}) \in E(G \boxtimes H)$ and is colored $d$. 
\\
\textbf{(b)} If $u \in R_{B} \setminus R_{A}$ then there exists $u_{1} \in R_{A}$ such that 
$(u, u_{1}) \in E(G \boxtimes H)$ and is colored $d$. 
\end{clm}
\begin{proof}
We consider the following cases.
\\
\\
\textit{\textbf{Case 1:}} \textbf{[When $u \in V_{0,z}$ where $0 \le z \le d(H)$, i.e $u \in V(H_{0})$ ]}
\\
In this case $u = [g_{0}, h_{j}] \in R_{A} \setminus R_{B}$.
We take $u_{1} = [g_{1}, h_{j}]$. Since $G$ is non-trivial, vertex $g_{1}$ exists and $(g_{0}, g_{1}) \in E(G)$.
Since $\ell_{G}(g_{i}) = 1$, we have
$u_{1} \in V_{1,z} \subseteq R_{B}$, where $1 \le z = \ell_{H}(h_{j}) \le d(H)$. Note that $z \ne 0$ since $u \ne [g_{0}, h_{0}]$.
Now the edge $(u, u_{1}) = ([g_{0}, h_{j}], [g_{1}, h_{j}]) \in E(G_{j})$. By \textit{Rule T$1$-R$4$},
$G_{j}$ is edge colored using the \textit{Layer-wise Coloring} $f_{G_{j}, A_{z}}$ or $f_{G_{j}, B_{z}}$, where $z = \ell_{H}(h_{j})$,
depending on whether $z$ is \textit{even} or \textit{odd}. 
Recalling that $A_{z} = \{a^{z}_{1}, a^{z}_{2}, \dots, a^{z}_{d(G)}\}$
and $B_{z} = \{b^{z}_{1}, b^{z}_{2}, \dots, b^{z}_{d(G)}\}$ the edge $(u, u_{1})$ is colored either $a^{z}_{1}$ or $b^{z}_{1}$.
Since $z \ge 1$, $a^{z}_{1} = b^{z}_{1} = d$ and hence the edge $(u, u_{1})$ is colored 
either $a^{z}_{1}=d$ or $b^{z}_{1}=d$.
\\
\\
\textit{\textbf{Case 2:}} \textbf{[When $u \in V_{1,z}$ where $0 \le z \le d(H)$]}
\\
In this case $u \in R_{B}$.
Note that if $z$ is odd then $V_{1,z} \subseteq R_{A} \cap R_{B}$. So we can assume that $z$ is even.
\\
\\
\textit{Case 2.a:} [When $u \in V_{1,0}$]
\\
Let $u = [g_{i}, h_{0}] \in V_{1,0}$ with $\ell_{G}(g_{i}) = 1$.
We take $u_{1} = [g_{0}, h_{1}]$. Since $H$ is non-trivial, $h_{1}$ exists and $(h_{0}, h_{1}) \in E(H)$.
Also edge $(g_{0}, g_{i}) \in E(G)$. Therefore the edge $(u, u_{1}) = ([g_{i}, h_{0}], [g_{0}, h_{1}]) \in E(G \boxtimes H)$.
It is easy to see that $u_{1} \in V_{0,1} \subseteq V(H_{0}) \subseteq R_{A}$. The edge $(u, u_{1})$ is colored
$d$ by \textit{Rule T$2$-R$3$}.
\\
\\
\textit{Case 2.b:} [When $u \in V_{1,z}$ where $2 \le z \le d(H)$ and $z$ is even]
\\
Let $u = [g_{i}, h_{j}] \in V_{1,z}$ with $\ell_{G}(g_{i}) = 1$. Then $(g_{0}, g_{1}) \in E(G)$.
We take $u_{1} = [g_{0}, h_{j}] \in V_{0,z} \subseteq R_{A}$, then $(u, u_{1}) = ([g_{i}, h_{j}], [g_{0}, h_{j}]) \in E(G_{j})$.
By \textit{Rule T$1$-R$4$}, $G_{j}$ is edge colored using the \textit{Layer-wise Coloring} $f_{G_{j}, A_{z}}$,
since $z = \ell_{H}(h_{j})$ is even. Since
$z \ge 2$, $a^{z}_{1} = b^{z}_{1} = d$ and the edge $(u , u_{1})$ is colored $d$.
\\
\\
\textit{\textbf{Case 3:}} \textbf{[When $u \in V_{y,z}$ where $2 \le y \le z$]}
\\
Let $u = [g_{i}, h_{j}] \in G_{j}$. Let $u_{1} = [g_{i'}, h_{j}]$ be the parent of $u$ in $G_{j}$. 
Since $\ell_{G}(g_{i'}) = \ell_{G}(g_{i}) - 1 = y-1$, $u_{1} \in V_{y-1,z}$. We claim that if 
$u \in V_{y,z} \subseteq R_{A} \setminus R_{B}$
then $u_{1} \in V_{y-1,z} \subseteq R_{B}$ 
\textbf{and} if $u \in V_{y,z} \subseteq R_{B} \setminus R_{A}$ then $u_{1} \in V_{y-1,z} \subseteq R_{A}$. 
To see this first note that
$\biguplus_{1 \le y \le z,\ |y-z|\ is\ even}V_{y,z} \subseteq R_{A}$ \textbf{and}
$\biguplus_{1 \le y \le z,\ |y-z|\ is\ odd}V_{y,z} \subseteq R_{B}$. Now the following is easy to see: if $2 \le y \le z$
and $V_{y,z} \subseteq R_{B} \setminus R_{A}\ (respectively\ R_{A} \setminus R_{B})$ 
then $1 \le y-1 < z$ and $V_{y-1,z} \subseteq R_{A}\ (respectively\ R_{B})$
since the parity of $|y-z|$ is different from the parity of $|(y-1)-z|$. 
By \textit{Rule T$1$-R$4$},
$G_{j}$ is edge colored using the \textit{Layer-wise Coloring} $f_{G_{j}, A_{z}}$ or $f_{G_{j}, B_{z}}$ 
depending on whether $z = \ell_{H}(h_{j})$ is \textit{even} or \textit{odd}. From the definition of the sets $A_{z}$
and $B_{z}$ we have that, for $1 \le i \le z$, $a^{z}_{i} = b^{z}_{i} = d$. Since $2 \le y \le z$, 
edge $(u, u_{1})$ is colored $a^{z}_{y} = d$ or $b^{z}_{y} = d$.
\\
\\
\textit{\textbf{Case 4:}} \textbf{[When $u \in V_{y,0}$ where $2 \le y \le d(G)$]}
\\
In this case $u \in R_{A} \setminus R_{B}$. Let $u = [g_{i}, h_{0}] \in V(H_{i})$. Let $u_{1} = [g_{i}, h_{1}] \in V(H_{i})$.
Since $(h_{0}, h_{1}) \in E(H)$, $(u, u_{1}) = ([g_{i}, h_{0}], [g_{i}, h_{1}]) \in E(H_{i})$. 
Vertex $u_{1} \in V_{y,1} \subseteq R_{B}$ as $(z=1) < 2 \le y$ and $1$ is \textit{odd}.
Since
$\ell_{G}(g_{i}) = y \ge 2$, by \textit{Rule T$1$-R$3$} all the edges of $H_{i}$ are colored $d$. Hence $(u, u_{1})$
is colored $d$.
\\
\\
\textit{\textbf{Case 5:}} \textbf{[When $u \in V_{y,z}$ where $1 \le z < y$]}
\\
Let $u = [g_{i}, h_{j}] \in V(H_{i})$. Let $u_{1} = [g_{i}, h_{j'}]$ be the parent of $u$ in $H_{i}$. Then
$(u, u_{1}) = ([g_{i}, h_{j}], [g_{i}, h_{j'}]) \in E(H_{i})$ and
$\ell_{H}(h_{j'}) = \ell_{H}(h_{j})-1 = z-1 \ge 0$. Since $y > z-1$ if 
$u \in V_{y,z} \subseteq R_{A} \setminus R_{B}\ (respectively\ R_{B} \setminus R_{A})$ then $z-1$ is \textit{odd} (\textit{even})
and $u \in V_{y,z-1} \subseteq R_{B}\ (respectively\ R_{A})$. Also since $y \ge 2$ by \textit{Rule T$1$-R$3$}, all the edges of $H_{i}$
are colored $d$.
\end{proof}
\begin{lem}
 \label{StrongfClaim}
The edge coloring $f$ is a rainbow coloring of $G \boxtimes H$.
\end{lem} 
\begin{proof} 
 We show that any distinct pair of vertices,
$u$ and $v$ from $G \boxtimes H$ have a rainbow path between them with respect to the edge coloring $f$.
Since $V(G \boxtimes H) = R_{A} \cup R_{B}$, vertex $u \in R_{A}$ or $u \in R_{B}$. The same applies to vertex $v$. 
Let $u = [g_{0}, h_{0}]$.
If $v \in R_{A}$ then by \textit{Claim }\ref{rar} there is an \textit{A-Rainbow-Path}
from $v$ to $u = [g_{0}, h_{0}]$. If $v \in R_{B}$ then by \textit{Claim }\ref{rbr} there is a \textit{B-Rainbow-Path}
from $v$ to some vertex $v' \in V_{1,0}$. We know that $(v',[g_{0}, h_{0}]) \in E(G_{0})$ and is colored $c$ by
the \textit{Layer-wise Coloring} $f_{G_{0},A_{0}}$. Hence there is a $(\{c\} \uplus B)$\textit{-Rainbow-Path}
from vertex $v$ to $u = [g_{0}, h_{0}]$.
\\
\indent
We may now assume that $u, v \ne [g_{0}, h_{0}]$. We have the following two cases:
\\
\\
\textit{\textbf{Case 1:}} \textbf{[When one of the vertices is in $R_{A}$ and the other is in $R_{B}$]}
\\
Without loss of generality let $u \in R_{A}$ and $v \in R_{B}$. By \textit{Claim }\ref{rar} there is an \textit{A-Rainbow-Path}
between vertex $u$ and vertex $[g_{0}, h_{0}]$, let this path be $P_{1}$. 
Similarly by \textit{Claim }\ref{rbr} there is a \textit{B-Rainbow-Path}
between vertex $v$ and some vertex $v_{1} = [g_{i}, h_{0}] \in V_{1,0}$, let this path be $P_{2}$. Now $v_{1} \in V(G_{0})$
and $\ell_{G}(g_{i}) = 1$, hence $(g_{0}, g_{1}) \in E(G)$ and $(v_{1}, [g_{0}, h_{0}]) \in E(G_{0})$. By \textit{Rule T$1$-R$4$} $G_{0}$
is edge colored using the \textit{Layer-wise Coloring} $f_{G_{0},A_{0}}$. The edge $(v_{1}, [g_{0}, h_{0}])$
is colored $a^{0}_{1} = c$. Clearly the path $P = P_{1}\ldotp ([g_{0}, h_{0}], v_{1})\ldotp P_{2}$ is a 
$(A \uplus B \uplus \{c\})$\textit{-Rainbow-Path}
between vertices $u$ and $v$.
\\
\\
\textit{\textbf{Case 2:}} \textbf{[When both the vertices are in $R_{A} \setminus R_{B}$]}
\\
By \textit{Claim }\ref{DLink} there exists a vertex $u_{1} \in R_{B} \subset V(G \boxtimes H)$ such that $(u, u_{1}) \in E(G \boxtimes H)$ and
is colored $d$.
Since $v \in R_{A}$ and $u_{1} \in R_{B}$ by \textit{Case 1} there is a $(A \uplus B \uplus \{c\})$\textit{-Rainbow-Path} 
from vertex $v$ to $u_{1}$, say $P_{1}$.
Clearly $P = P_{1} \ldotp (u_{1}, u)$ is a rainbow path from vertex $v$ to vertex $u$.
\\
\\
\textit{\textbf{Case 3:}} \textbf{[When both the vertices are in $R_{B} \setminus R_{A}$]}
\\
By \textit{Claim }\ref{DLink} there exists a vertex $u_{2} \in R_{A} \subset V(G \boxtimes H)$ such that $(u, u_{2}) \in E(G \boxtimes H)$ and
is colored $d$. Now using arguments similar to \textit{Case 2} we can prove that there exists a rainbow path between vertices
$u$ and $v$.
\end{proof}
\begin{thm}
\label{StrongTheorem}
 $r(G' \boxtimes H') \le rc(G' \boxtimes H') \le 2r(G' \boxtimes H') + 2$
\end{thm}
\begin{proof}
 The rainbow coloring $f$ uses $|A| + |B| + |\{c, d\}| = 2d(G) + 2 = 2r(G \boxtimes H) + 2$ colors.
Since $d(G) = r(G') = r(G' \boxtimes H')$
From
of \textit{Claim} \ref{StrongfClaim} the upper bound follows. The lower bound is trivial.
\end{proof}

\noindent
\textbf{Tight Example:}
\\
Consider two graphs $G_{1}$ and $G_{2}$ such that $diam(G_{1})$ = $2r(G_{1}) \ge diam(G_2)$. For example $G_{1}$ may be taken as a path with odd number of vertices. Then $rc(G_{1} \boxtimes G_{2}) \ge diam(G_{1} \boxtimes G_{2}) = 2r(G_{1} \boxtimes G_2)$.

\bibliographystyle{plain}
\bibliography{rainbow}

\begin{thebibliography}{10}

\bibitem{basavaraju2010rainbow}
M.~Basavaraju, L.S. Chandran, D.~Rajendraprasad, and A.~Ramaswamy.
\newblock {Rainbow connection number and radius}.
\newblock {\em Arxiv preprint arXiv:1011.0620v1 [math.CO]}, 2010.

\bibitem{caro2008rainbow}
Y.~Caro, A.~Lev, Y.~Roditty, Z.~Tuza, and R.~Yuster.
\newblock {On rainbow connection}.
\newblock {\em The electronic journal of combinatorics}, 15(R57):1, 2008.

\bibitem{chakraborty2009hardness}
Sourav Chakraborty, Eldar Fischer, Arie Matsliah, and Raphael Yuster.
\newblock Hardness and algorithms for rainbow connection.
\newblock {\em Journal of Combinatorial Optimization}, pages 1--18, 2009.

\bibitem{chandran2010raindom}
L.~Sunil Chandran, Anita Das, Deepak Rajendraprasad, and Nithin~M. Varma.
\newblock {Rainbow Connection Number and Connected Dominating Sets}.
\newblock {\em Arxiv preprint arXiv:1010.2296v1 [math.CO]}, 2010.

\bibitem{chandran2011raincon}
L.~Sunil Chandran, Xueliang Li, Sujuan Liu, Rogers Mathew, and Deepak
  Rajendraprasad.
\newblock Rainbow connection number and connectivity.
\newblock {\em Submitted}, 2011.

\bibitem{chartrand2008rainbow}
G.~Chartrand, G.~L. Johns, K.~A. McKeon, and P.~Zhang.
\newblock {Rainbow connection in graphs}.
\newblock {\em Math. Bohem}, 133(1):85--98, 2008.

\bibitem{chartrand2008chromatic}
G.~Chartrand and P.~Zhang.
\newblock {\em {Chromatic graph theory}}.
\newblock Chapman \& Hall, 2008.

\bibitem{dong2011rainbow}
J.~Dong and X.~Li.
\newblock Rainbow connection number, bridges and radius.
\newblock {\em Arxiv preprint arXiv:1105.0790}, 2011.

\bibitem{gologranc2011rainbow}
T.~Gologranc and G.~Meki{\v{s}}.
\newblock Rainbow connection and graph products.
\newblock {\em IMFM preprint}, 2011.

\bibitem{imrich2000product}
W.~Imrich, S.~Klav{\v{z}}ar, and B.~Gorenec.
\newblock {\em Product graphs: structure and recognition}.
\newblock Wiley, 2000.

\bibitem{klavzar2011rainbow}
S.~Klavzar and G.~Meki{\v{s}}.
\newblock On the rainbow connection of cartesian products and their subgraphs.
\newblock {\em IMFM preprint}, 2011.

\bibitem{krivelevich2010rainbow}
M.~Krivelevich and R.~Yuster.
\newblock {The rainbow connection of a graph is (at most) reciprocal to its
  minimum degree}.
\newblock {\em Journal of Graph Theory}, 63(3):185--191, 2010.

\bibitem{li2011characterize}
X.~Li and Y.~Sun.
\newblock Characterize graphs with rainbow connection number m-2 and rainbow
  connection numbers of some graph operations.
\newblock 2011.

\bibitem{li2010rain3con}
Xueliang Li and Yongtang Shi.
\newblock Rainbow connection in 3-connected graphs.
\newblock {\em Arxiv preprint arXiv:1010.6131v1 [math.CO]}, 2010.

\end{thebibliography}

\end{document}